\documentclass[letter,11pt]{article}

\usepackage[T1]{fontenc}
\usepackage[linktocpage=true]{hyperref}

\usepackage[margin=1.15 in, top=.9 in, bottom=0.9 in]{geometry}

\makeatletter
\g@addto@macro\normalsize{%
  \setlength\abovedisplayskip{7pt}
  \setlength\belowdisplayskip{7pt}
  \setlength\abovedisplayshortskip{7pt}
  \setlength\belowdisplayshortskip{7pt}
}
\makeatother
\usepackage{tocloft}

\usepackage[utf8]{inputenc}
\usepackage{amsfonts}
\usepackage{mathrsfs}
\usepackage{bbm}
\usepackage{latexsym}
\usepackage{amssymb}
\usepackage{mathtools}
\usepackage{lipsum}
\usepackage[normalem]{ulem}

\usepackage{enumitem}
\setlist{nolistsep}
\usepackage{amsthm}
\usepackage[capitalize]{cleveref}

\newtheoremstyle{plain}{3mm}{3mm}{\slshape}{}{\bfseries}{.}{.5em}{}
\newtheoremstyle{definition}{2mm}{2mm}{}{}{\bfseries}{.}{.5em}{}
\newtheoremstyle{inproofclaim}{2mm}{2mm}{}{}{\bfseries}{.}{.5em}{}
\theoremstyle{plain}
	
\newtheorem{Theorem}{Theorem}
\newtheorem{Lemma}[Theorem]{Lemma}
\newtheorem{lemma}[Theorem]{Lemma}
\newtheorem{Proposition}[Theorem]{Proposition}
\newtheorem{Corollary}[Theorem]{Corollary}

\newtheorem{Question}[Theorem]{Question}
\newtheorem{Conjecture}[Theorem]{Conjecture}	
\theoremstyle{definition}
\newtheorem{Definition}[Theorem]{Definition}
\newtheorem{Remark}[Theorem]{Remark}
\newtheorem{Example}[Theorem]{Example}
\theoremstyle{inproofclaim}

\theoremstyle{plain}
\newcounter{MainTheoremCounter}

\newtheorem{Maintheorem}[MainTheoremCounter]{Theorem}

\theoremstyle{plain}
\newtheorem*{namedthm}{\namedthmname}
\newcounter{namedthm}
\makeatletter
	\newenvironment{named}[2]
	{\def\namedthmname{#1}
	\refstepcounter{namedthm}
	\namedthm[#2]\def\@currentlabel{#1}}
	{\endnamedthm}
\makeatother

\usepackage{chngcntr}
\counterwithin{Theorem}{section}

\numberwithin{equation}{section}

\allowdisplaybreaks

\usepackage{xcolor}
\definecolor{Color1}{rgb}{0.0, 0.42, 0.47}
\definecolor{Scarlet}{rgb}{0.78, 0.11, 0.0}
\definecolor{Color3}{rgb}{0.39, 0.71 ,0.0}
\hypersetup{citecolor = black,colorlinks,
			linkcolor = black,
			urlcolor = Scarlet}

\usepackage{titlesec}
\titleformat{\section}
  {\Large\center\bfseries}
  {\thesection.}{.7em}{}
\titlespacing*{\section}{0pt}{3.5ex plus 0ex minus 0ex}{1.5ex plus 0ex}
\titleformat{\subsection}
  {\large\center\bfseries}
  {\thesubsection.}{.7em}{}
\titlespacing*{\subsection}{0pt}{3.5ex plus 0ex minus 0ex}{1.5ex plus 0ex}
\titleformat{\subsubsection}
  {\center\bfseries}
  {\thesubsubsection.}{.7em}{}
\titlespacing*{\subsubsection}{0pt}{3.5ex plus 0ex minus 0ex}{1.5ex plus 0ex}

\newcommand{\Cesaro}{Ces\`{a}ro}

\newcommand{\Szemeredi}{Szemer\'{e}di}

\newcommand{\supp}{{\normalfont\text{supp}}\,}
\newcommand{\Oh}{{\rm O}}
\newcommand{\oh}{{\rm o}}
\newcommand{\N}{\mathbb{N}}
\newcommand{\Z}{\mathbb{Z}}
\newcommand{\R}{\mathbb{R}}

\newcommand{\Q}{\mathbb{Q}}
\newcommand{\T}{\mathbb{T}}
\newcommand{\Hilb}{\mathscr{H}}
\newcommand{\Cont}{C}

\newcommand{\define}[1]{{\itshape #1}}
\newcommand{\lhk}{\lvert\!|\!|}
\newcommand{\rhk}{|\!|\!\rvert}

\renewcommand{\epsilon}{\varepsilon}
\renewcommand{\leq}{\leqslant}
\renewcommand{\geq}{\geqslant}
\renewcommand{\setminus}{\backslash}
\renewcommand{\Re}{{\rm Re}}

\renewcommand{\phi}{\varphi}

\newcommand{\E}{\mathbb{E}}
\newcommand{\B}{\mathcal{B}}

\renewcommand{\d}{~\mathrm{d}}

\newcommand{\Hardy}{\mathcal{H}}

\newcommand{\F}{\mathcal{F}}
\newcommand{\dom}{\mathsf{dom}}
\newcommand{\upperdens}{\overline{\mathrm{d}}}
\renewcommand{\sp}{\mathrm{span}^*}
\newcommand{\sps}{\mathrm{span}}

\newcommand{\Hsp}{\nabla\text{-}\mathrm{span}}
\newcommand{\poly}{\mathrm{poly}}

\author{Vitaly~Bergelson \and Joel~Moreira \and Florian~K.~Richter\footnote{The third author is supported by the National Science Foundation under grant number DMS~1901453.}}
\date{\small \today}
\title{\bfseries Multiple ergodic averages along functions from a Hardy field: convergence, recurrence and combinatorial applications}

\begin{document}

\maketitle
\begin{abstract}
We obtain new results pertaining to convergence and recurrence of multiple ergodic averages along functions from a Hardy field. Among other things, we confirm some of the conjectures posed by Frantzikinakis in \cite{Frantzikinakis10,Frantzikinakis16} and obtain combinatorial applications which contain, as rather special cases, several previously known (polynomial and non-polynomial) extensions of \Szemeredi{}'s theorem on arithmetic progressions \cite{BL96,BLL08,FW09,Frantzikinakis10,BMR20}.
One of the novel features of our results, which is not present in previous work, is that they allow for a mixture of polynomials and non-polynomial functions.
As an illustration, assume $f_i(t)=a_{i,1}t^{c_{i,1}}+\cdots+a_{i,d}t^{c_{i,d}}$ for $c_{i,j}>0$ and $a_{i,j}\in\R$.
Then
\begin{itemize}
 \item for any measure preserving system $(X,\B,\mu,T)$ and $h_1,\dots,h_k\in L^\infty(X)$, the limit
$$\lim_{N\to\infty}\frac{1}{N}\sum_{n=1}^N T^{[f_1(n)]}h_1\cdots T^{[f_k(n)]}h_k$$
exists in $L^2$;
    \item for any $E\subset \N$ with $\upperdens(E)>0$ there are $a,n\in\N$ such that $\{a,\, a+[f_1(n)],\ldots,a+[f_k(n)]\}\subset E$.
\end{itemize}
We also show that if $f_1,\dots,f_k$ belong to a Hardy field, have polynomial growth, and are such that no linear combination of them is a polynomial, then for any measure preserving system $(X,\B,\mu,T)$ and any $A\in\B$,
$$\limsup_{N\to\infty}\frac{1}{N}\sum_{n=1}^N\mu\Big(A\cap T^{-[ f_1(n) ]}A\cap\ldots\cap T^{-[f_k(n)]}A\Big)\,\geq\,\mu(A)^{k+1}.$$

\end{abstract}

\small
\tableofcontents
\thispagestyle{empty}
\normalsize


\section{Introduction}
\label{sec_intro}


The goal of this paper is to establish a strong multiple recurrence theorem which has results obtained in \cite{BL96,BLL08,FW09,Frantzikinakis10,BMR20} as special cases and produces new applications to combinatorics.
In particular, it provides a solution to an open problem posed by Frantzikinakis \cite[Problem 25]{Frantzikinakis16}, and allows us to obtain partial progress on another \cite[Problem 23]{Frantzikinakis16}.

\subsection{Combinatorial results}
\label{sec_combi_results}

The \define{upper density} of a set $E\subset \N$ is defined as
$
\upperdens(E)\,=\,\limsup_{N\to\infty}{|E\cap \{1,\ldots,N\}|}/{N}.
$
One of the central themes in Ramsey theory is the study of arithmetic patterns that appear in large sets of natural numbers. 
In particular, one would like to know for which sequences $g_1,\ldots,g_k\colon \N\to \N$ one can always find a configuration of the form
\begin{equation}
\label{eqn_gen_patterns}
\{a,\, a+g_1(n), \ldots, a+ g_k(n)\}
\end{equation}
in any set $E\subset \N$ of positive upper density.
A fundamental result in this direction is the following theorem of E.\ \Szemeredi{}.

\begin{Theorem}[\Szemeredi{}'s Theorem, \cite{Szemeredi75}]{}
\label{thm_szmeredi}
For any set $E\subset \N$ of positive upper density and any $k\in\N$ there exist $a,n\in\N$ such that $\{a,\,a+n,\ldots,a+kn\}\subset E$.
\end{Theorem}

An extension of \Szemeredi{}'s Theorem dealing with the case of \eqref{eqn_gen_patterns} where $g_1,\ldots,g_k$ are polynomials was obtained in \cite{BL96}. The following theorem pertains to the one-dimensional case of this result.

\begin{Theorem}[Polynomial \Szemeredi{} Theorem, \cite{BL96}]{}
\label{thm_poly_szmeredi}
For any set $E\subset \N$ of positive upper density and any polynomials $q_1,\ldots,q_k\in\Z[t]$ satisfying $q_1(0)=\ldots=q_k(0)=0$ there exist $a,n\in\N$ such that $\{a,\, a+q_1(n),\ldots,a+q_k(n)\}\subset E$.
\end{Theorem}

\cref{thm_poly_szmeredi} was later improved in \cite{BLL08} to give an ``if and only if'' condition. 
A finite collection of polynomials $q_1,\ldots,q_k\in\Z[t]$ is called \define{jointly intersective} if for all $m\in\N$ there is $n\in\N$ such that $q_1(n)\equiv\ldots \equiv q_k(n)\equiv {0}\bmod{m}$.

\begin{Theorem}[\cite{BLL08}]
\label{thm_poly_szmeredi_iff}
Given $q_1,\ldots,q_k\in\Z[t]$ the following are equivalent:
\begin{enumerate}
[label=(\roman{enumi}),ref=(\roman{enumi}),leftmargin=*]
\item
The polynomials $q_1,\ldots,q_k$ are jointly intersective.
\item
For any set $E\subset \N$ of positive upper density there exist $a,n\in\N$ such that $\{a,\, a+q_1(n),\ldots,a+q_k(n)\}\subset E$.
\end{enumerate}
\end{Theorem}


In view of the above results, one is led to inquire whether similar results hold for more general classes of, say, eventually monotone sequences that do not grow too fast\footnote{We remark in passing that for any sequence of \emph{exponential} growth $q(n)$ there exists a set $E\subset\N$ with $\upperdens(E)>0$ which contains no pair of the form $\{x,x+q(n)\}$ with $x,n\in\N$ (cf. 
\cite[Corollary 4.18]{BBHS08}).}.
A natural class of sequences to consider are those arising from Hardy fields.


\begin{Definition}
Let $\mathsf{G}$ denote the ring (under pointwise addition and multiplication) of germs at infinity\footnote{A \define{germ
at infinity} is any equivalence class of real-valued functions in one real variable under the equivalence relationship
$(f\equiv g) \Leftrightarrow \big(\exists t_0>0
~\text{such that}~f(t)=g(t)~\text{for all}~t\in [t_0,\infty)\big)$.} of real valued functions defined
on a half-line $[s,\infty)$ for some $s\in\R$.
Any subfield of $\mathsf{G}$ that is closed under differentiation is called a \define{Hardy field}.
\end{Definition}

By abuse of language, we say that a function $f\colon[s,\infty)\to\R$ belongs to some Hardy field $\Hardy$, and write $f\in\Hardy$, if its germ at infinity belongs to $\Hardy$.
{For convenience, we assume all Hardy fields considered in this paper are \define{shift-invariant}, i.e., if $f\in \Hardy$ then $f_u\in\Hardy$ for all $u\in\R$, where $f_u(t)=f(t+u)$.}
A classical example of a {(shift-invariant)} Hardy field is the class of \define{logarithmico-exponential functions} introduced by Hardy in \cite{Hardy12,Hardy10}. It consists of all (germs of) real-valued functions that can be built from real polynomials, the logarithmic function $\log(t)$, and the exponential function $\exp(t)$ using the standard arithmetical operations $+$, $-$, $\cdot$, $\div$ and the operation of composition.
Examples of logarithmico-exponential functions are ${p(t)}/{q(t)}$ for $p(t),q(t)\in \R[t]$, $t^c$ for $c\in \R$, ${t}/{\log (t)}$ and $e^{\sqrt{t}}$, as well as any products or linear combinations of the above.
Other Hardy fields contain even more exotic functions, such as $e^{\sqrt{\log t}}\Gamma(t)$ and $t^c\zeta(t)$ for any $c\in\R$, where $\zeta$ is the Riemann zeta function and $\Gamma$ is the usual gamma function (cf.~\cite{Boshernitzan84a}).
For more information on Hardy fields we refer the reader to \cite{Boshernitzan81,Boshernitzan94, Frantzikinakis09}.

Another analogue of \Szemeredi{}'s Theorem, which involves functions from a Hardy field, is due to Frantzikinakis \cite{Frantzikinakis15b} (see also \cite{FW09,Frantzikinakis10}).
Throughout the paper we use $\lfloor.\rfloor\colon\R\to\Z$ to denote the \define{floor function}, i.e., for all $x\in\R$ the expression $\lfloor x \rfloor$ stands for the largest integer less than or equal to $x$.

\begin{Theorem}[\cite{Frantzikinakis15b}; cf.\ also {\cite[Theorem 2.10]{Frantzikinakis10}}]
\label{thm_frantzi}
Let $f_1,\ldots,f_k$ be functions from a {Hardy field} such that for every $f\in\{f_1,\ldots,f_k\}$ there is $\ell\in\N$ such that ${f(t)}/{t^{\ell}}\to 0$ and ${t^{\ell-1}\log(t)}/{f(t)}\to 0$ as $t\to\infty$ and have \define{different growth}, in the sense that for all $i\neq j$ either ${f_{i}(t)}/{f_{j}(t)}\to 0$ or ${f_{j}(t)}/{f_{i}(t)}\to 0$. Then for any set $E\subset \N$ of positive upper density there exist $a,n\in\N$ such that $\{a,\, a+\lfloor f_1(n)\rfloor ,\ldots,a+\lfloor f_k(n)\rfloor \}\subset E$.
\end{Theorem}

Another variant of \Szemeredi{}'s Theorem, which was recently obtained by the authors in \cite{BMR20}, stands in general position to \cref{thm_frantzi}
and reveals a new phenomenon pertaining to multiple recurrence along a wide family of non-polynomial (and not necessarily Hardy) functions.
The following is one of the combinatorial corollaries of the main result in \cite{BMR20}.

\begin{Theorem}
\label{thm_thick_szemeredi}
Let $f$ be a function from a \define{Hardy field} and assume there is $\ell\in\N$ such that ${f(t)}/{t^{\ell}}\to 0$ and ${t^{\ell-1}}/{f(t)}\to 0$ as $t\to\infty$. Then for  any set $E\subset \N$ of positive upper density there exist $a,n\in\N$ such that $\{a,\, a+\lfloor f(n)\rfloor,a+\lfloor f(n+1)\rfloor,\ldots,a+\lfloor f(n+k)\rfloor \}\subset E$.
\end{Theorem}

Both Theorems \ref{thm_frantzi} and \ref{thm_thick_szemeredi} leave some space for further inquiry.
For instance, 
one would like to know if \cref{thm_thick_szemeredi} has a version for several functions from $\Hardy$ which deals with the patterns
$$\{a,\, a+[f_1(n)],\ldots,a+[f_1(n+\ell)],\ldots,a+[f_k(n)],\ldots,a+[f_k(n+\ell)]\}\subset E.$$
As for \cref{thm_frantzi}, a natural generalization is addressed by the following conjecture of Frantzikinakis.
Given a finite collection of functions $f_1,\ldots,f_k$, define
\[
\sps(f_1,\ldots,f_k) \coloneqq \big\{c_1 f_1(t)+\ldots+c_k f_k(t): (c_1,\ldots,c_k)\in\R^k\big\},
\]
and
\[
\sp(f_1,\ldots,f_k) \coloneqq \big\{c_1 f_1(t)+\ldots+c_k f_k(t): (c_1,\ldots,c_k)\in\R^k\setminus\{0\}\big\}.
\]
Also, we will write $f(t)\prec g(t)$ when ${g(t)}/{f(t)}\to\infty$ as $t\to\infty$, and $f(t)\ll g(t)$ when there exist $C>0 $ and $t_0\geq 1$ such that $f(t)\leq C g(t)$ for all $t\geq t_0$.  
We say $f(t)$ has \define{polynomial growth} if it satisfies $|f(t)|\ll t^d$ for some $d\in\N$. 

\begin{Conjecture}[see {\cite[Problems 4 and 4']{Frantzikinakis10}} and {\cite[Problem 25]{Frantzikinakis16}}]
\label{conj_frantzi}
Let $f_1,\ldots,f_k$ be functions of polynomial growth from a Hardy field such that $$|f(t)-q(t)|\to\infty$$ for all $f\in\sp(f_1,\dots,f_k)$ and $q\in\Z[t]$. Then for any set $E\subset \N$ of positive upper density there exist $a,n\in\N$ such that $\{a,\, a+\lfloor f_1(n)\rfloor,\ldots,a+\lfloor f_k(n)\rfloor \}\subset E$.
\end{Conjecture}

\begin{Remark}
\label{rem_other_rounding_functions}
In the statements of Theorems \ref{thm_frantzi} and \ref{thm_thick_szemeredi} and \cref{conj_frantzi}, it is possible to replace the floor function $\lfloor.\rfloor\colon\R\to\Z$ with other rounding functions, such as the \define{ceiling function} $\lceil\cdot \rceil\colon\R\to\Z$, or the \define{rounding to the closest integer function} $[.]\colon\R\to\Z$.
Indeed, since $\lceil x\rceil=-\lfloor -x\rfloor$ and $[x]=\lfloor x+ 0.5\rfloor$, replacing $\lfloor.\rfloor$ with either $\lceil . \rceil$ or $[.]$ actually yields equivalent formulations of those statements.
\end{Remark}

The following theorem is the main combinatorial result of this paper. It contains Theorems~\ref{thm_szmeredi}, \ref{thm_poly_szmeredi}, \ref{thm_poly_szmeredi_iff}, \ref{thm_frantzi},
\ref{thm_thick_szemeredi} as special cases and confirms \cref{conj_frantzi}.
We denote by $\poly(f_1,\ldots,f_k)$ the set of all real polynomials that can be ``generated'' using linear combinations of the functions $f_1,\ldots,f_k$.
More precisely,
$$
\poly(f_1,\ldots,f_k)\,=\,\left\{p\in\R[t]: \exists f\in\sps(f_1,\dots,f_k)\,\text{with}\, \lim_{t\to\infty}|f(t)-p(t)|=0\right\}.
$$

\begin{Maintheorem}
\label{thm_main_combinatorial}
Let $f_1,\ldots,f_k$ be functions of polynomial growth from a Hardy field and assume that at least one of the following two conditions holds:
\begin{enumerate}
[label=(\arabic{enumi}),ref=(\arabic{enumi}),leftmargin=*]
\item
\label{itm_cond_2}
For all $q\in \Z[t]$ and $f\in\sp(f_1,\ldots,f_k)$ we have $\lim_{t\to\infty}|f(t)-q(t)|= \infty$.
\item
\label{itm_cond_1}
There is a jointly intersective collection of polynomials $q_1,\ldots,q_\ell\in\Z[t]$ such that
\vspace{-.3em}
$$
\poly(f_1,\dots,f_k)\subset \sps(q_1,\dots,q_\ell).
$$
\end{enumerate}
\vspace{-.2em}
Then for any set $E\subset \N$ of positive upper density there exist $a,n\in\N$ such that $\{a,\, a+[f_1(n)],\ldots,a+[f_k(n)]\}\subset E$.
\end{Maintheorem}

\begin{Remark}\label{rmrk_roundingfunction}
Condition \ref{itm_cond_2} in Theorem~\ref{thm_main_combinatorial} will still hold if all the $f_i$ are shifted.
Together with \cref{rem_other_rounding_functions}, this observation implies that if we are in case \ref{itm_cond_2} of Theorem~\ref{thm_main_combinatorial} then the conclusion of Theorem~\ref{thm_main_combinatorial} remains true, even if the closest integer function $[\cdot]$ is replaced by either $\lfloor . \rfloor$ or $\lceil.\rceil$ (or indeed with any other rounding function of the form $\lfloor x+c\rfloor$, $c\in[0,1)$). 
This observation, however, does not apply to condition \ref{itm_cond_1} -- see \cref{example_1} below.
\end{Remark}

Besides resolving \cref{conj_frantzi}, Theorem~\ref{thm_main_combinatorial} also implies numerous new results. Corollaries A1, A2, A3, and A4 below comprise a selection of such results that we consider to be of particular interest.


The following rather special case of Theorem~\ref{thm_main_combinatorial} already implies Theorems \ref{thm_szmeredi}, \ref{thm_poly_szmeredi}, \ref{thm_frantzi}, and \ref{thm_thick_szemeredi}.

\begin{named}{Corollary A1}{}
Let $f_1,\ldots,f_k$ be functions of polynomial growth from a Hardy field and assume that every $p\in \poly(f_1,\ldots,f_k)$ satisfies $p(0)=0$.
Then for any set $E\subset \N$ of positive upper density there exist $a,n\in\N$ such that $\{a,\, a+[f_1(n)],\ldots,a+[f_k(n)]\}\subset E$.
\end{named}

The fact that any set of positive density contains an arrangement of the form
$$
\{a,\, a+[n^{c_1}],\ldots,a+[n^{c_k}]\},
$$
where $c_1,\ldots,c_k$ are all positive integers, follows from \cref{thm_poly_szmeredi}, whereas the case when $c_1,\ldots,c_k$ are all positive non-integers follows from \cref{thm_frantzi}. The next corollary of Theorem~\ref{thm_main_combinatorial} deals with the previously unknown case where the constants $c_1,\ldots,c_k$ are a mix of integers and non-integers.

\begin{named}{Corollary A2}{}
For any $c_1,\ldots,c_k>0$ and any set $E\subset \N$ of positive upper density there exist $a,n\in\N$ such that $\{a,\, a+[n^{c_1}],\ldots,a+[n^{c_k}]\}\subset E$. The same is true with $[.]$ replaced by either $\lfloor.\rfloor$ or $\lceil.\rceil$.
\end{named}

Here are examples of other configurations that were not covered by previously known results:
$$
\{a,\, a+[ n^{c_1}],a+[ n^{c_1}+n^{c_2}]\},~\text{or}~\{a,\, a+n, a+[n^{c}]\},~\text{or}~\{a,\,a+[ \log(n)],  a+[n^{c}]\}.
$$
The following corollary takes care of these and more general configurations.
We denote by $\log_m(t)$ the $m$-th iterate of $\log(t)$, that is, $\log_1(t)=\log(t)$, $\log_2(t)=\log\log(t)$, $\log_3(t)=\log\log\log(t)$, and so on.
Let $\mathcal{K}$ denote the smallest algebra of functions that contains $t^c$ for all $c>0$ and $\log_m^r(t)$ for all $m\in\N$ and $r>0$.
We stress that any polynomial $p$ in $\mathcal{K}$ satisfies $p(0)=0$ (in particular, $\mathcal{K}$ doesn't contain non-zero constant functions). Therefore, from Corollary A1, we obtain the following clean statement.

\begin{named}{Corollary A3}{}
For any $f_1,\ldots,f_k\in\mathcal{K}$ and $E\subset \N$ with $\upperdens(E)>0$ there are $a,n\in\N$ such that $\{a,\, a+[f_1(n)],\ldots,a+[f_k(n)]\}\subset E$.
\end{named}

We remark that Corollaries A1 and A3 are not true if $[.]$ is replaced by either $\lfloor.\rfloor$ or $\lceil.\rceil$ (see \cref{example_1} below).

While Corollary A3 contains \cref{thm_poly_szmeredi} as a special case, it does not encompass polynomials which have a non-zero constant term.
The following theorem shows that the conclusion of Corollary A3 holds for significantly more general families ${\mathcal K}$.
In particular, in conjunction with Theorem~\ref{thm_main_combinatorial}, it implies both \cref{thm_poly_szmeredi_iff} and \cref{thm_frantzi}.

\begin{Theorem}\label{thm_110}
Let $q\in\Z[t]$ be an intersective polynomial, let $\Hardy$ be a Hardy field and let ${\mathcal L}\subset\Hardy$ be a family of functions such that any $f\in{\mathcal L}$ satisfies $t^{k-1}\prec f(t)\prec t^k$ for some $k\in\N$, and any distinct $f,g\in{\mathcal L}$ have different growth.
Let ${\mathcal K}({\mathcal L},q)$ be the linear span over $\R$ of ${\mathcal L}$ and $q\R[t]$.
Then any tuple $f_1,\dots,f_k$ from ${\mathcal K}({\mathcal L},q)$ satisfies Condition \ref{itm_cond_1} of Theorem~\ref{thm_main_combinatorial}.
\end{Theorem}

Throughout this work, we use  $f^{(m)}(t)$ to denote the $m$-th derivative of a function $f(t)$.
Also, given a finite set of functions $f_1,\ldots,f_k$, define
\begin{align*}
\Hsp(f_1,\ldots,f_k)&
\coloneqq \sps\big(\{f_i^{(m)}:1\leq i\leq k;\ m\geq0\}\big)
\end{align*}
Theorem~\ref{thm_main_combinatorial} allows us to derive a corollary which extends \cref{thm_thick_szemeredi} from a single function $f$ to multiple functions $f_1,\ldots,f_k
$.


\begin{named}{Corollary A4}{}
Let $f_1,\ldots,f_k$ be functions of polynomial growth from a Hardy field and assume that for all $f\in\Hsp(f_1,\ldots,f_k)$ we have $\lim_{t\to\infty}|f(t)|\in\{0, \infty\}$.
Then for any $\ell\in\N$, any set $E\subset \N$ of positive upper density contains a configuration of the form
$$\{a,\, a+[f_1(n)],a+[f_1(n+1)],\ldots,a+[f_1(n+\ell)],\ldots,a+[f_k(n)],a+[f_k(n+1)],\ldots,a+[f_k(n+\ell)]\}.$$
The same is true with $[.]$ replaced by either $\lfloor.\rfloor$ or $\lceil.\rceil$.
\end{named}

The assumptions of Corollary A4 are satisfied, for instance, if all the $f_i$ are linear combinations of powers $t^c$ with non-integer exponents $c>0$.
On the other hand, if some $f\in\sp(f_1,\dots,f_k)$ is a polynomial, then the conclusion of Corollary A4 fails.
In Example~\ref{example5} below we show that the conclusion may fail even when $\poly(f_1,\dots,f_k)=\emptyset$.

\subsection{Ergodic results}
\label{sec_erg_results}

In \cite{Furstenberg77} Furstenberg developed an ergodic approach to Szemer\'edi's theorem, thereby establishing a connection between dynamics and additive combinatorics.
The quintessence of Furstenberg's method is captured by the following theorem.


\begin{Theorem}[Furstenberg's correspondence principle, see {\cite[Theorem 1.1]{Bergelson87}} and \cite{Bergelson96}]\label{thm_correspondence}
For any $E\subset\N$ with $\upperdens(E)>0$ there exists an invertible measure preserving system $(X,\B,\mu,T)$ and a set $A\in\B$ with $\mu(A)=\upperdens(E)$ such that for all $n_1,\ldots,n_\ell\in\Z$,
\begin{equation}
    \label{eq_FCP}
    \upperdens\big(E\cap (E-n_1)\cap\ldots\cap (E-n_\ell)\big)\,\geq\, \mu\big(A\cap T^{-n_1}A\cap\ldots\cap T^{n_\ell}A\big).
\end{equation}

\end{Theorem}

In light of \eqref{eq_FCP}, it is natural to tackle Theorem~\ref{thm_main_combinatorial} by studying the behaviour of Ces\`aro averages of
multicorrelation expressions of the form
\begin{equation}\label{eq_idk}
 \alpha(n):=\mu\big(A\cap T^{-[f_1(n)]}A\cap \ldots\cap T^{-[f_k(n)]}A\big).
\end{equation}
Unfortunately, the limit of the averages $\frac1N\sum_{n=1}^N\alpha(n)$ does not exist in general.
This is, for example, the case when some of the $f_i$ grow too slowly.
Nevertheless, this issue can be overcome by considering weighted ergodic averages of the form

\begin{equation}
\label{eqn_multi_erg_thm}
\frac{1}{W(N)}\sum_{n=1}^N w(n)\mu\big(A\cap T^{-[f_1(n)]}A\cap \ldots\cap T^{-[f_k(n)]}A\big),
\end{equation}
where $W\colon\N\to\R_{>0}$ is a non-decreasing sequence satisfying $\lim_{n\to\infty} W(n)=\infty$ and $w(n)\coloneqq \Delta W(n)= W(n+1)-W(n)$ is its \define{discrete derivative}.

There are many choices of $W$ for which the limit as $N\to\infty$ of the averages \eqref{eqn_multi_erg_thm} exists.
We say that a weight function $W$ belonging to a Hardy field $\Hardy$ is \emph{compatible} with the functions $f_1,\dots,f_k\in\Hardy$ if $1\prec W(t)\ll t$ and the following holds:
\vspace{.3 em}

\begin{enumerate}
[label=\textbf{Property~(P):},ref=(P),leftmargin=*]
\item\label{property_P_W}
For all $f\in\Hsp(f_1,\ldots,f_k)$ and $p\in\R[t]$ either $|f(t)-p(t)|\ll 1$ or \\$|f(t)-p(t)| \succ \log(W(t))$.
\end{enumerate}
\vspace{.3 em}

\noindent It is important to mention that given a finite collection of functions $f_1,\ldots,f_k\in\Hardy$ of polynomial growth from a Hardy field $\Hardy$, not every $W\in\Hardy$ will satisfy property \ref{property_P_W}.
For example $W(t)=t$ is not compatible with $f(t)=\log t$.
On the other hand, there always exists a compatible $W\in\Hardy$ 
(see \cref{cor_finding_W}).
Roughly speaking, $W$ can be chosen as any function which tends to $\infty$ and grows ``slower'' than any unbounded function in
$$\Hsp(f_1,\dots,f_k)-\R[x]=\big\{f-g:f\in \Hsp(f_1,\dots,f_k), g\in\R[x]\big\}.$$

The following theorem is the second main result of this paper.

\begin{Maintheorem}
\label{thm_main}
Let $f_1,\ldots,f_k$ be functions of polynomial growth from a Hardy field $\Hardy$ and let $W\in\Hardy$ be any function
with $1\prec W(t)\ll t$ and such that $f_1,\ldots,f_k$ satisfy Property \ref{property_P_W}.
Let $(X,\B,\mu,T)$ be an invertible measure preserving system.
\begin{enumerate}
[label=(\roman{enumi}),ref=(\roman{enumi}),leftmargin=*]
\item
\label{itm_thm_main_i}
For any $h_1,\ldots,h_k\in L^\infty(X)$ the limit
\begin{equation}
\label{eqn_multiple_recurrence_closest-integer}
\lim_{N\to\infty}\frac{1}{W(N)}\sum_{n=1}^N w(n) T^{[f_1(n)]}h_1\cdots T^{[f_k(n)]}h_k
\end{equation}
exists in $L^2$.
\item
\label{itm_thm_main_iii}
If $f_1,\ldots,f_k$ satisfy condition \ref{itm_cond_2} of Theorem~\ref{thm_main_combinatorial} then for any $h_1,\ldots,h_k\in L^\infty(X)$
\begin{equation}
\label{eqn_multiple_recurrence_closest-integer_product}
\lim_{N\to\infty}\frac{1}{W(N)}\sum_{n=1}^N w(n) T^{[f_1(n)]}h_1\cdots T^{[f_k(n)]}h_k\,=\, \prod_{i=1}^k h_i^*\quad\text{in}~L^2,
\end{equation}


where $h_i^*$ is the orthogonal projection of $h_i$ onto the subspace of $T$-invariant functions in $L^2(X)$.
In particular, for any $A\in{\mathcal B}$
$$
\lim_{N\to\infty}\frac{1}{W(N)}\sum_{n=1}^N w(n) \mu(A\cap T^{-[f_1(n)]}A\cap\cdots\cap T^{-[f_k(n)]}A)\geq\mu(A)^{k+1}.
$$
\item
\label{itm_thm_main_ii}
If $f_1,\ldots,f_k$ satisfy condition \ref{itm_cond_1} of Theorem~\ref{thm_main_combinatorial} then for any $A\in\B$ with $\mu(A)>0$ we have
\begin{equation}
\label{eqn_multiple_recurrence_floor}
\lim_{N\to\infty}\frac{1}{W(N)}\sum_{n=1}^N w(n) \mu\Big(A\cap T^{-[ f_1(n) ]}A\cap\ldots\cap T^{-[f_k(n)]}A\Big)\,>\,0.
\end{equation}
\end{enumerate}
\end{Maintheorem}

As a first corollary to Theorem~\ref{thm_main} we obtain convergence of multiple ergodic averages with \Cesaro{} weights along a rather large class of functions from a Hardy field.
Special cases of this corollary were previously obtained in
\cite{HK05a,HK05b,Leibman05a,BK09,Frantzikinakis10,Koutsogiannis18}.

\begin{named}{Corollary B1}{}
Let $f_1,\ldots,f_k$ be functions of polynomial growth from a Hardy field $\Hardy$ such that for all $f\in\Hsp(f_1,\ldots,f_k)$ and $p\in\R[t]$ either $|f(t)-p(t)|\ll 1$ or $|f(t)-p(t)| \succ \log(t)$. Then for any ergodic invertible measure preserving system $(X,\B,\mu,T)$ and any $h_1,\ldots,h_k\in L^\infty(X)$ the limit
\begin{equation*}
\lim_{N\to\infty}\frac{1}{N}\sum_{n=1}^N T^{[f_1(n)]}h_1\cdots T^{[f_k(n)]}h_k
\end{equation*}
exists in $L^2$. The same is true with $[.]$ replaced by either $\lfloor.\rfloor$ or $\lceil.\rceil$.
\end{named}

Theorem~\ref{thm_main} also provides additional information on the following result, which was originally conjectured by Frantzikinakis ({\cite[Problem 23]{Frantzikinakis16}}; cf.\ also {\cite[Problem 1]{Frantzikinakis15b}}) and recently proved by Tsinas \cite{Tsinas23}.

\begin{Theorem}[{\cite{Tsinas23}}]
\label{conj_frantzi_2}
Let $f_1,\ldots,f_k$ be functions from a Hardy field such that
$|f(t)-q(t)|/\log t\to\infty$ for every $q\in\Z[t]$ and $f\in\sp(f_1,\ldots,f_k)$. Then for any ergodic measure preserving system $(X,\B,\mu, T)$ and any $f_1,\ldots,f_k\in L^\infty(X)$ we have
\begin{equation*}
\lim_{N\to\infty}\frac{1}{N}\sum_{n=1}^N  T^{[f_1(n)]}h_1\cdots T^{[f_k(n)]}h_k\,=\,\prod_{i=1}^k \int h_i\d\mu\quad\text{in $L^2$.}
\end{equation*}
\end{Theorem}

Part \ref{itm_thm_main_iii} of Theorem~\ref{thm_main} implies \cref{conj_frantzi_2} under the additional assumption that the functions $f_1,\dots,f_k$ satisfy Property~\ref{property_P_W} with $W(t)=t$, which is slightly stronger than the assumption that $|f(t)-q(t)|/\log t\to\infty$ for every $q\in\Z[t]$ and $f\in\sp(f_1,\ldots,f_k)$. However, Theorem~\ref{thm_main} gives more than that. Indeed, part \ref{itm_thm_main_iii} of Theorem~\ref{thm_main} provides a confirmation of a variant of
\cref{conj_frantzi_2} where the condition $|f(t)-q(t)|/\log t\to\infty$ is weakened to $|f(t)-q(t)|\to\infty$, but at the price of replacing Ces\`aro averages by weighted averages (where one can choose any weight for which Property~\ref{property_P_W} is satisfied, cf.~\cref{cor_finding_W}).

\begin{named}{Corollary B2}{}
\label{cor_B2}
Let $f_1,\ldots,f_k$ be functions of polynomial growth from a Hardy field $\Hardy$ and let $W\in\Hardy$ be any function
with $1\prec W(t)\ll t$ and such that $f_1,\ldots,f_k$ satisfy Property \ref{property_P_W}.
Suppose that for all $q\in \Z[t]$ and $f\in\sp(f_1,\ldots,f_k)$ we have $\lim_{t\to\infty}|f(t)-q(t)|= \infty$.
Then for any invertible ergodic measure preserving system $(X,\B,\mu,T)$ and any $h_1,\dots,h_k\in L^\infty(X)$
\begin{equation*}
\lim_{N\to\infty}\frac{1}{W(N)}\sum_{n=1}^N w(n) T^{[f_1(n)]}h_1\cdots T^{[f_k(n)]}h_k\,=\, \prod_{i=1}^k \int h_i\d\mu\quad\text{ in $L^2$.}
\end{equation*}
The same is true when $[.]$ is replaced by either $\lfloor.\rfloor$ or $\lceil.\rceil$.
\end{named}

\begin{Example}\label{example_rightlimit}
  For each $i=1,\dots,k$ let $f_i$ be of the form $f_i(t)=a_1t^{c_1}+\cdots+a_dt^{c_d}$ for $a_i\in\R$ and $c_i>0$, $c_i\notin\Z$ and assume that the $f_i$ are linearly independent.
  Then $W(t)=t$ is a compatible weight and hence, in view of Corollary B2, we deduce that for any invertible ergodic measure preserving system $(X,\B,\mu,T)$ and any $h_1,\dots,h_k\in L^\infty(X)$
\begin{equation}
\lim_{N\to\infty}\frac{1}N\sum_{n=1}^N T^{[f_1(n)]}h_1\cdots T^{[f_k(n)]}h_k\,=\, \prod_{i=1}^k \int h_i\d\mu\quad\text{ in $L^2$.}
\end{equation}
\end{Example}
We remark that in the case when $f_1,\ldots,f_k$ have different growth, this result was obtained by Frantzikinakis in \cite[Theorem 2.6]{Frantzikinakis10}


Finally, we formulate two corollaries of Theorem~\ref{thm_main} which imply Theorem~\ref{thm_main_combinatorial} (via Furstenberg's correspondence principle).
We believe that these two results are also of independent interest.

\begin{named}{Corollary B3}{}
\label{cor_B3}
Let $f_1,\ldots,f_k$ be functions of polynomial growth from a Hardy field $\Hardy$ satisfying condition \ref{itm_cond_2} of Theorem~\ref{thm_main_combinatorial}.
Then for any invertible measure preserving system $(X,\B,\mu,T)$ and any $A\in\B$,
\begin{equation}\label{eq_CorB3}
\limsup_{N\to\infty}\frac{1}{N}\sum_{n=1}^N\mu\Big(A\cap T^{-[ f_1(n) ]}A\cap\ldots\cap T^{-[f_k(n)]}A\Big)\,\geq\,\mu(A)^{k+1}.
\end{equation}
The same is true when $[.]$ is replaced by either $\lfloor.\rfloor$ or $\lceil.\rceil$.
\end{named}

\begin{named}{Corollary B4}{}
\label{cor_B4}
Let $f_1,\ldots,f_k$ be functions of polynomial growth from a Hardy field $\Hardy$ satisfying condition \ref{itm_cond_1} of Theorem~\ref{thm_main_combinatorial}.
Then for any invertible measure preserving system $(X,\B,\mu,T)$ and any $A\in\B$ with $\mu(A)>0$ we have
\begin{equation}\label{eq_CorB4}
    \limsup_{N\to\infty}\frac{1}{N}\sum_{n=1}^N\mu\Big(A\cap T^{-[ f_1(n) ]}A\cap\ldots\cap T^{-[f_k(n)]}A\Big)\,>\,0.
\end{equation}
\end{named}

We would like to stress that one cannot replace $\limsup$ with $\lim$ in \eqref{eq_CorB3} and \eqref{eq_CorB4}, as the limits don't exist in general.
In fact, if one replaces $\limsup$ with $\liminf$, the left hand side may equal $0$ in either case.


\subsection{Illustrative examples and counterexamples}

While part \ref{itm_thm_main_i} of Theorem~\ref{thm_main} holds for \emph{any} functions $f_1,\dots,f_k\in\Hardy$ of polynomial growth (with an appropriate $W$), parts \ref{itm_thm_main_iii} and \ref{itm_thm_main_ii} require some additional constraints.
This is of course unavoidable since for certain functions from $\Hardy$ the conclusions of parts \ref{itm_thm_main_iii} and \ref{itm_thm_main_ii} of Theorem~\ref{thm_main} are known to fail.
In this subsection we collect examples that illustrate that under some ostensibly natural weakenings/modifications of its conditions, parts \ref{itm_thm_main_iii} and \ref{itm_thm_main_ii} of Theorem~\ref{thm_main} are no longer true.

As was discussed in \cref{rmrk_roundingfunction}, if condition \ref{itm_cond_2} of Theorem~\ref{thm_main_combinatorial} holds, then the conclusion remains valid when the floor function $\lfloor.\rfloor$ is replaced with either the ceiling function $\lceil . \rceil$ or the closest integer function $[.]$. 
However, as the next example demonstrates, the same is not true in general for condition \ref{itm_cond_1} of Theorem~\ref{thm_main_combinatorial}.

\begin{Example}
\label{example_1}
Let $c>0$ be such that $n^c\notin\N$ for any $n\in\{2,3,...\}$ and consider the pair of functions
\[
f_1(n)= n-n^c
\qquad\text{and}\qquad f_2(n)= n+n^c.
\]
Since $f_1,f_2$ satisfy condition \ref{itm_cond_1} of Theorem~\ref{thm_main_combinatorial}, it follows that any set $E\subset\N$ of positive upper density contains a triple of the form $\{a,\, a+[f_1(n)],a+[ f_2(n)]\}$. However, the set $E=2\N-1$ does not contain any configurations of the form $\{a,\, a+\lfloor f_1(n)\rfloor,a+\lfloor f_2(n)\rfloor\}$ or $\{a,\, a+\lceil f_1(n)\rceil,a+\lceil f_2(n)\rceil\}$.
\end{Example}


One may wonder if in Theorem~\ref{thm_main_combinatorial} one can relax condition \ref{itm_cond_2} by replacing $\sp(f_1,\ldots,f_k)$ with $\sp_\Z(f_1,\ldots,f_k)=\{c_1f_1(t)+\ldots+c_kf_k(t): c_1,\ldots,c_k\in \Z\}$. However, the following example shows that this is not possible.

\begin{Example}
\label{example_2}
Let $\alpha$ and $\beta$ be two rationally independent irrationals with $\max\{|\alpha|,|\beta|\}\leq 1/8$, and consider the functions
\[
f_1(n)= \alpha^{-1} n^2+n
\qquad\text{and}\qquad f_2(n)= \beta^{-1}(n^3-\alpha n+1/2).
\]
Then the set $E=\{n\in\N: \{n\alpha\}\in [0,1/8),\,\{n\beta\}\in [0,1/8)\}$ does not contain a triple of the form $\{a,\, a+[f_1(n)],a+ [f_2(n)]\}$.

To see why this is the case, observe that $\{a,\, a+[f_1(n)]\}\subset E$ implies that the number $[f_1(n)]\alpha$ is within $\frac{1}{8}$ of an integer, which when combined with the condition $|\alpha|\leq1/8$ implies that $f_1(n)\alpha$ is within $\frac{1}{4}$ of an integer.
Similarly, $\{a,\, a+[f_2(n)]\}\subset E$ implies that $f_2(n)\beta$ is within $\frac{1}{4}$ of an integer.
On the other hand, due to the choice of $f_1$ and $f_2$, for every $n\in\N$ we have $f_1(n)\alpha+f_2(n)\beta\equiv1/2\bmod1$; so which precludes that $\{a,\, a+[f_1(n)],a+ [f_2(n)]\}\subset E$.

\end{Example}

Next, we address the following natural question.

\begin{Question}\label{question_intro}
Let $\Hardy$ be a Hardy field.
Is it true that for any $f_1,\dots,f_k\in\Hardy$ with polynomial growth the set of returns
\begin{equation}
\label{eqn_R_A}
R=\Big\{n\in\N: \mu\big(A\cap T^{-[f_1(n)]}A\cap\ldots\cap T^{-[f_k(n)]}A\big)>0\Big\},
\end{equation}
is either finite or satisfies $\bar d(R)>0$ (or at least $d^*(R)>0$\footnote{Given a set $R\subset\N$ the \emph{upper Banach density} of $R$ is defined by
$$d^*(R):=\lim_{N\to\infty}\sup_{M\in\N}\frac1N|R\cap\{M,M+1,\dots,M+N\}|.$$})?
\end{Question}
The following example shows that the answer is negative even when all the functions involved are essentially polynomials.
\begin{Example}\label{example_8}
Let $\alpha$ be an irrational, $C>0$,
and consider the functions
\[
f_1(n)= 2n\alpha-1/2
\qquad\text{and}\qquad f_2(n)= 2n\alpha+1/2-2C/n.
\]
Moreover, let $T\colon \{0,1\}\to\{0,1\}$ be the map $T(x)=x+1\bmod 2$ and $A=\{0\}$. A straightforward calculation shows that
$$
\Big\{n\in\N: A\cap T^{-[f_1(n)]}A\cap T^{-[f_2(n)]}A\neq\emptyset\Big\}=\big\{n\in\N: \{n\alpha-1/2\}<C/n\big\},
$$
which, for $C$ sufficiently small, is an infinite set of zero upper Banach density by the classical Tchebychef's inhomogenous version of Dirichlet's Approximation Theorem (see \cite{Grace18}).
\end{Example}
\begin{Remark}
Examples analogous to \cref{example_8} exist where $[\cdot]$ is replaced by $\lfloor\cdot\rfloor$ and $\lceil\cdot\rceil$.
One can also create an example where the dynamical system is an irrational rotation instead of a rotation on only two points.
Finally we point out that, by choosing $\alpha$ appropriately one can make the set $\big\{n\in\N: \{n\alpha-1/2\}<C/n\big\}$ be arbitrarily sparse.
\end{Remark}

In view of \cref{example_8} one might suspect that polynomials with irrational coefficients are the only obstruction to an affirmative answer to \cref{question_intro}.
The following example shows that this is not the case.
In fact, no integer linear combination of the functions $f_1$ and $f_2$ in the following example belongs to $\R[x]\setminus\Z[x]$.
\begin{Example}\label{example_4}
  Let $f_1(n)=n+\sqrt{n}$ and $f_2(n)=n-\sqrt{n}$.
  Since $\lfloor f_{1}(n)\rfloor $ and $\lfloor f_{2}(n)\rfloor $ have different parity whenever $n$ is not a perfect square, we see that, for a rotation on two points (i.e. with the same $T$ and $A$ as in \cref{example_8}) the set \eqref{eqn_R_A} has $0$ Banach upper density.
  However, it follows from \cref{thm_poly_szmeredi} that the set \eqref{eqn_R_A} has an infinite intersection with the set of perfect squares, and hence is itself infinite.
\end{Example}

Our final example shows that the hypothesis in Corollary A4 can not be weakened to the assumption that $\poly(f_1,\dots,f_k)=\emptyset$.
\begin{Example}\label{example5}
Let $f_1(t)=t^{5/2}$ and let $f_2(t)=5/2t^{3/2}+t$.
Let $\alpha\in\R$ be irrational, let $\epsilon>0$ be small and let $E=\{n\in\N:\{n\alpha\}<\epsilon\}$.
Since the linear combination $$t\mapsto f_1(t+2)-2f_1(t+1)+f_1(t)-f_2(t+1)+f_2(t)$$ tends to $1$ as $t\to\infty$, if for some $a,n\in\N$ we have $a+\{0,[f_1(n)],[f_1(n+1)],f_1(n+2)],[f_2(n)],[f_2(n+1)]\}\subset E$, then we would have $\{k\alpha\}<6\epsilon$ for some integer $k$ with $|k|\leq6$. By choosing $\epsilon$ and $\alpha$ appropriately, this becomes impossible, showing that the conclusion of Corollary A4 fails for these functions.
\end{Example}
\subsection{Outline of the paper}
After presenting some preliminaries in \cref{sec_prelims}, in \cref{sec_proofs} we briefly explain how each result mentioned in the introduction follows from Theorem~\ref{thm_main}.
Sections \ref{sec_charfac} and \ref{sec_nilfac} are dedicated to the proof of Theorem~\ref{thm_main}.
In \cref{sec_charfac} we show that the nilfactors are \emph{characteristic} for the expressions involved, and in \cref{sec_nilfac} we use this to reduce Theorem~\ref{thm_main} to the case of nilsystems, which can then be studied directly using equidistribution results developed in \cite{Richter20arXiv}.
Finally, in \cref{sec_questions}, we formulate some natural open questions.

\paragraph{Acknowledgments.}We thank N.~Frantzikinakis for providing helpful comments on an earlier version of this paper.
We thank Saúl Rodríguez Martín for pointing out an error in the definition of $\Hsp(f_1,\ldots,f_k)$ in the version of this paper published in Advances in Mathematics; the error has been corrected in the present arXiv version.

%
%

\section{Preliminaries}
\label{sec_prelims}
\subsection{Preliminaries on nilsystems and nilmanifolds}
\label{sec_perlims_nilsystems}

Let $G$ be a ($s$-step) nilpotent Lie group and $\Gamma$ a \define{uniform}\footnote{A closed subgroup $\Gamma$ of a Lie group $G$ is called \define{uniform} if the quotient space $G/\Gamma$, endowed with the quotient topology, is a compact topological space.} and \define{discrete}\footnote{A subgroup $\Gamma$ of a Lie group $G$ is called discrete if there exists an open cover of $\Gamma$ in which every open set contains exactly one element of $\Gamma$.} subgroup of $G$. The quotient space $X\coloneqq G/\Gamma$ is called a \define{($s$-step) nilmanifold}.
An example of a ($1$-step) nilmanifold is $X=\T^d\coloneqq\R^d/\Z^d$.

An element $g\in G$ with the property that $g^n\in \Gamma$
for some $n\in \N$ is called \define{rational} (or \define{rational with respect to $\Gamma$)}.
A closed subgroup $H$ of $G$ is then called \define{rational}
(or \define{rational with respect to $\Gamma$}) if rational elements are dense in $H$.
For instance, if $G=\R^2$ and $\Gamma=\Z^2$ then the subgroup $H=\{(t,\alpha t):t\in\R\}$ is rational if and only if $\alpha\in\Q$.

Rational subgroups play a key role in the description of sub-nilmanifolds.
If $X=G/\Gamma$ is a nilmanifold, then a \define{sub-nilmanifold} $Y$ of $X$ is any
closed set of the form $Y=Hx$, where $x\in X$ and $H$ is a closed
subgroup of $G$.
It is not true that for every closed subgroup $H$ of $G$ and
every element $x=g\Gamma$ in $X=G/\Gamma$ the set $Hx$ is a sub-nilmanifold of $X$, because $Hx$ need not be closed. In fact, it is shown in \cite{Leibman06} that $Hx$ is closed in $X$ (and hence a sub-nilmanifold) if and only if the subgroup $g^{-1}Hg$ is rational with respect to $\Gamma$.
For more information on rational elements and rational subgroups see \cite{Leibman06}.

The Lie group $G$ acts naturally on $X$ via left-multiplication given by the formula $a(g\Gamma)=(ag)\Gamma$ for every $a\in G$ and $g\Gamma\in x$.
There exists a unique Borel probability measure on $X$ that is invariant under this action by $G$ called the \define{Haar measure} on $X$ (see \cite{Raghunathan72}), which we denote by $\mu_X$.

By a \define{($s$-step) nilsystem} we mean a pair $(X,T)$ where $X$ is a ($s$-step) nilmanifold and $T\colon X\to X$ is left-multiplication by a fixed element $a\in G$.
Since the Haar measure $\mu_X$ is $T$-invariant, a nilsystem is simultaneously a topological dynamical system and a (probability) measure preserving system.
It is well known that for a nilsystem, being transitive, minimal and uniquely ergodic are all equivalent properties.

We denote by $G^\circ$ the connected component of $G$ that contains the identity element $1_G$ of $G$.
If $a\in G^\circ$, then $a^t$ is well defined for every $t\in\R$ and the nilsystem $(X,T)$ can be naturally extended to a flow (i.e. a $\R$-action on $X$ whose time $1$ map is $T$).

\begin{Proposition}\label{prop_malcev}

Given a minimal nilsystem $(X,T)$ there exists a nilsystem $(Y,S)$ with $Y=G/\Gamma$ for a connected and simply connected nilpotent Lie group $G$ such that $(X,T)$ is (conjugate to) a subsystem of $(Y,S)$.
Moreover, the flow induced by $(Y,S)$ is ergodic and if $\pi:G\to Y$ is the natural projection and $H\subset G$ is a closed and rational subgroup such that $X=\pi(H)$ then $\Gamma\subset H$.
\end{Proposition}

We thank A. Leibman for help with the following proof.

\begin{proof}
Let $X=G'/\Gamma'$ and let $a\in G'$ be such that $Tx=ax$ for all $x\in X$.
Let $G'_c$ denote the connected component of the identity in $G'$, and let $\pi':G'\to X$ be the natural projection.
Then $\pi(G'_c)$ is open in $X$ and thus, by minimality, $\pi(a^\Z G_c')=a^\Z\pi(G_c')=X$.
Let $G'':=a^\Z G_c'$, let $\Gamma'':=\Gamma'\cap G''$ and denote by $G''_c$ the connected component of the identity in $G''$.
Then we have $X=G''/\Gamma''$ and $a\in G''$, but since $G''_c=G'_c$ we also have $G''=a^\Z G''_c$.

Since $X$ is compact it has a finite number $d$ of connected components. 
For any $x\in X$ the point $a^dx$ is in the same connected component as $x$.
This implies that there exists $\gamma_0\in\Gamma''$ such that $a^d\gamma_0\in G''_c$ and $d$ is the smallest natural number with this property.
Therefore $\Gamma''$, being contained in $G''=a^\Z G''_c$, is generated by $\Gamma''\cap G_c''$ and $\gamma_0$.

Using \cite[Theorem 5.1.6]{CG90} we can find a Malcev basis, i.e. generating set $\{\gamma_1,\dots,\gamma_k\}$ for $\Gamma''\cap G_c''$ such that $G_c''=\{\gamma_1^{t_1}\cdots\gamma_k^{t_k}:t_1,\dots,t_k\in\R\}$.
Therefore $\Gamma''$ is generated by $\{\gamma_0,\gamma_1,\dots,\gamma_k\}$ and
$$G''=a^\Z G_c''=\big\{\gamma_0^{t_0}\cdots\gamma_k^{t_k}:t_0\in\tfrac1d\Z,t_1,\dots,t_k\in\R\big\}.$$


Let $s\in\N$ be the nilpotency step of $G''$, let $F$ be the free product of $k+1$ copies of $\R$, and let $F'$ be the free product of $\Z$ and $k$ copies of $\R$, embedded as a subgroup of $F$ in the natural way.
There is a surjective homomorphism $\psi:F'\to G''$ taking each generating copy of $\R$ to the one parameter subgroups $\{\gamma_i^t:t\in\R\}$ of $G''$, and taking the copy of $\Z$ to $\{\gamma_0^n:n\in\Z\}$.
Since $G''$ is $s$-step nilpotent, $\psi$ must vanish on the $(s+1)$-st group $F'_{s+1}$ in the lower central series of $F'$.
One can in fact show that $\psi$ must vanish on $F_{s+1}\cap F'$, where $F_{s+1}$ is the $(s+1)$-st group in the lower central series of $F$, and hence $\psi$ descends to a homomorphism on the group $G:=F'/(F_{s+1}\cap F')$.
Let $\hat G:=F/F_{s+1}$, let $\tilde G$ be the connected component of the identity in $G$, let $\Gamma=\psi^{-1}(\Gamma'')$ and let $\hat a\in G$ be such that $\psi(\hat a)=a$.
The four listed properties can now be checked by a routine argument.
\end{proof}



\subsection{Preliminaries on Hardy fields}
\label{sec_perlims_hardy}

Let $W\in\Hardy$ with $1\prec W(t)\ll t$ and recall that $w(n)\coloneqq \Delta W(n)= W(n+1)-W(n)$.
Given a nilmanifold $X$ and a Borel probability measure $\nu$ on $X$, we say a sequence $(x_n)_{n\in\N}$ of points in $X$ is \define{uniformly distributed  with respect to $\nu$ and $W$-averages} if
\begin{equation}
\label{eqn_def_ud}
\lim_{N\to\infty}\frac{1}{W(N)}\sum_{n=1}^N w(n) F(x_n) = \int F\d\nu
\end{equation}
holds for every continuous function $F\in\Cont(X)$.
This notion extends the classical notion of uniform distribution $\bmod \,1$, which corresponds to the case where $W(N)=N$, the nilmanifold $X=\R/\Z$ is the one dimensional torus and $\nu$ is the Lebesgue measure.

We need the following two results form \cite{Richter20arXiv}.

\begin{Lemma}[{\cite[Corollary A.5]{Richter20arXiv}}]
\label{cor_finding_W}
Let $\Hardy$ be a Hardy field and assume $f_1,\ldots,f_k\in \Hardy$ have polynomial growth. Then there exists $W\in\Hardy$ with $1\prec W(t)\ll t$ such that $f_1,\ldots,f_k$ satisfy property \ref{property_P_W}.
\end{Lemma}

\begin{Lemma}[{\cite[Theorem 5.1]{Richter20arXiv}}]
\label{prop_Bosh_W-averages}
Let $\Hardy$ be a Hardy field and $W\in\Hardy$ be a function satisfying $1\prec W(t)\ll t$. Then for any $f\in\Hardy$ with the property that $t^{\ell-1}\log(W(t))\prec |f(t)| \prec t^\ell$ for some $\ell\in\N$, the sequence $(f(n))_{n\in\N}$ is uniformly distributed mod $1$ with respect to $W$-averages.
\end{Lemma}

We also need the following Lemma.
\begin{Lemma}
\label{lem_simple_normal_form_new_new}
Assume $f_1,\ldots,f_k\in \Hardy$ have polynomial growth.
Then we can partition $\{1,\ldots,k\}$ into two sets $\mathcal{I}$ and $\mathcal{J}$ such that
\begin{enumerate}
[label=(\alph{enumi}),ref=(\alph{enumi}),leftmargin=*]
\item
\label{itm_Z1}
Any $f\in\sp\{f_j: j\in{\mathcal J}\}$ satisfies $|f(t)-p(t)|\to\infty$ for any $p\in\R[t]$
\item
\label{itm_Z2}
for any $i\in\mathcal{I}$ there exist $p_i\in\poly(f_1,\dots,f_k)$, and $\{\lambda_{i,j}:j\in\mathcal{J}\}\subset\R$ such that
$$
\lim_{t\to\infty}\Bigg|f_i(t)-\sum_{j\in\mathcal{J}} \lambda_{i,j} f_j(t)-p_i(t)\Bigg|=0.
$$
\end{enumerate}
\end{Lemma}

\begin{proof}
We use induction on $k$. For the base case of this induction, which corresponds to $k=1$, we distinguish between the cases when $\lim_{t\to\infty}|f_1(t)-p(t)|<\infty$ for some $p\in\R[t]$, and when
$\lim_{t\to\infty}| f_1(t)-p(t)|=\infty$ for all $p\in\R[t]$ (noting that the limit must exist since any Hardy field function is eventually monotone.).
If we are in the first case then simply take $\mathcal{I}=\{1\}$, $\mathcal{J}=\emptyset$, $c=\lim_{t\to\infty}f_1(t)-p(t)$, and $p_1(t)=p(t)+c$. If the latter holds, then we pick $\mathcal{I}=\emptyset$ and $\mathcal{J}=\{1\}$ and we are done.

Now assume the claim has already been proven for $k-1$.
This means that for the collection $\{f_1,\ldots,f_{k-1}\}$  there exist disjoint $\mathcal{I}',\mathcal{J}'\subset\{1,\ldots,k-1\}$ with $\mathcal{I}'\cup\mathcal{J}'=\{1,\ldots,k-1\}$, $\{\lambda_{i,j}: i\in\mathcal{I}', j\in\mathcal{J}'\}\subset\R$, and $\{p_i: i\in\mathcal{I}'\}\subset\poly(f_1,\dots,f_{k-1})$ such that conditions \ref{itm_Z1} and \ref{itm_Z2} are satisfied. We again distinguish two cases.
The first case is when there exists $\{\eta_j: j\in\mathcal{J}'\}\subset\R$, $\eta_k\in\R$, and $p\in\R[t]$ such that $\lim_{t\to\infty}|\sum_{j\in\mathcal{J}'}\eta_j f_j(t)+ \eta_k f_k(t)-p(t)|<\infty$.
If we are in this case, then we take $\mathcal{I}=\mathcal{I}'\cup\{k\}$, $\mathcal{J}=\mathcal{J}'$, $c=\lim_{t\to\infty}\sum_{j\in\mathcal{J}'}\eta_j f_j(t)+ \eta_k f_k(t)-p(t)$, $p_k(t)=\eta_k^{-1}(p(t)+c)$, and $\lambda_{k,j}=-\eta_k^{-1}\eta_j$ for all $j\in\mathcal{J}'$. It is then straightforward to check that conditions \ref{itm_Z1} and \ref{itm_Z2} are satisfied.
The second case is when $\lim_{t\to\infty}|\sum_{j\in\mathcal{J}'}\eta_j f_j(t)+ \eta_k f_k(t)-p(t)|=\infty$ for all $\{\eta_j: j\in\mathcal{J}'\}\subset\R$, $\eta_k\in\R$, and $p\in\R
[t]$. But this just means that if we choose $\mathcal{I}=\mathcal{I}'$ and $\mathcal{J}=\mathcal{J}'\cup\{k\}$ then conditions \ref{itm_Z1} and \ref{itm_Z2} hold.
\end{proof}

We also need the following corollary of \cref{prop_Bosh_W-averages}.


\begin{Corollary}
\label{Cor_2claims}
Let $f\in\Hardy$ and assume $f$ satisfies property \ref{property_P_W} with respect to some $W\in\Hardy$ with $1\prec W(t)\ll t$. If
\begin{equation}
\label{eqn_acc_at_one-half}
    \lim_{\epsilon\to0}d_W(\{n:\|f(n)-1/2\|_\T<\epsilon\})>0,
\end{equation}
    then $f=p+g$ where $p\in\Q[x]$ and $\lim_{t\to\infty} g(t)\to 0$.
\end{Corollary}

\begin{proof}
Since $f$ satisfies property \ref{property_P_W}, we either have $|f(t)-p(t)|\ll 1$ for some $p\in\R[t]$ or $|f(t)-p(t)| \succ \log(W(t))$ for all $p\in\R[t]$.
If we are in the latter case then we claim that $f(n)$ is uniformly distributed mod $1$ with respect to $W$-averages, which makes \eqref{eqn_acc_at_one-half} impossible.
To prove the claim, let $\ell$ be the smallest integer for which $f(t)\ll t^\ell$.
If $\ell>1$ then a standard inductive argument using the version of van der Corput's trick\footnote{Strictly speaking, Corollary 2.6 in \cite{BMR20} is formulated for well distribution (an amplified variant of uniform distribution), but the exact same derivation (using \cite[Proposition 2.5]{BMR20}) holds for the version we use here.} in \cite[Corollary 2.6]{BMR20} reduces the claim to the case $\ell=1$.
For $\ell=1$ the claim follows by \cref{prop_Bosh_W-averages}.

So we must be in the first case, when $|f(t)-p(t)|\ll 1$ for some $p\in\R[t]$.
Define $c=\lim_{t\to\infty} f(t)-p(t)$.
By replacing $p(t)$ with $p(t)+c$, we can assume without loss of generality that $c=0$.
Note that $f(t)=g(t)+p(t)$ where $g(t)$ is some function that satisfies $\lim_{t\to\infty} g(t)\to 0$.  If $p-p(0)\notin \Q[x]$ then in view of Weyl's Equidistribution Theorem \cite{Weyl16} the sequence $f(n)$ is uniformly distributed mod $1$ with respect to \Cesaro{} averages and hence also uniformly distributed mod $1$ with respect to $W$-averages, making \eqref{eqn_acc_at_one-half} impossible. Therefore we must have $p-p(0)\in\Q[x]$. Finally, note that $p(0)$ must be rational, because otherwise all the accumulation points of $f(n)$ mod~$1$ are irrationals, contradicting \eqref{eqn_acc_at_one-half}.
\end{proof}

\section{Proofs of the corollaries}
\label{sec_proofs}

In this section we explain how all the results in the introduction can be derived from Theorem~\ref{thm_main}.

Corollary B1 follows directly from part \ref{itm_thm_main_i} of Theorem~\ref{thm_main} by taking $W(t)=t$ and Corollary B2 follows directly from part \ref{itm_thm_main_iii} of Theorem~\ref{thm_main} in the case of an ergodic system.
Corollaries B3 and B4 follow from parts \ref{itm_thm_main_iii} and \ref{itm_thm_main_ii}, respectively, together with the fact that for any bounded sequence $a:\N\to\R$ taking non-negative values we have\footnote{
For completeness, we provide a proof of \eqref{eq_weightedineqaulity}:
Let $c_N(M)=\frac M{W(N)}\left(w(M)-w(M+1)\right)$ if $M<N$, and $c_N(N)=\frac{Nw(N)}{W(N)}$.
Since $w(n)$ is eventually non-increasing, $c_N(M)\geq0$ for sufficiently large $M$. 
The following identity, which can be readily checked by induction, shows that the weighted averages of $a$ are related to the unweighted averages via the coefficients $c_N(M)$:
$$\frac1{W(N)}\sum_{n=1}^Nw(n)a(n)=\sum_{M=1}^Nc_N(M)\biggl(\frac1M\sum_{n=1}^Ma(n)\biggr).$$
Applying this formula with $a\equiv1$ it follows that, for each $N\in\N$, the coefficients $c_N(1),\dots,c_N(N)$ add up to~$1$.
Therefore, for an arbitrary non-negative $a:\N\to\R$, for any $N$ there exists some $M=M(N)\leq N$ such that
$$\frac1M\sum_{n=1}^Ma(n)\geq\frac1{W(N)}\sum_{n=1}^Nw(n)a(n).$$
Moreover, since $c_N(M)\to0$ as $N\to\infty$ for any fixed $N$, we have that $M(N)\to\infty$ as $N\to\infty$, so we get \eqref{eq_weightedineqaulity}.
}
\begin{equation}
\label{eq_weightedineqaulity}\limsup_{N\to\infty}\frac1N\sum_{n=1}^Na(n)\geq\limsup_{N\to\infty}\frac1{W(N)}\sum_{n=1}^Nw(n)a(n).
\end{equation}

Theorem~\ref{thm_main_combinatorial} is obtained by combining Corollaries B3 and B4 with Furstenberg's correspondence principle (\cref{thm_correspondence}).

\begin{proof}[Proof of Corollary A1]
Since all the $f_i$ have polynomial growth, there is a maximum degree $d$ among all polynomials in $\poly(f_1,\dots,f_k)$.
If all $p\in\poly(f_1,\dots,f_k)$ satisfies $p(0)=0$, then $\poly(f_1,\dots,f_k)\subset\sps(t,t^2,\dots,t^d)$. 
Since the polynomials $t,t^2,\dots,t^d$ are jointly intersective, the result follows from part \ref{itm_cond_1} of Theorem~\ref{thm_main_combinatorial}.
\end{proof}

\begin{proof}[Proof of Corollary A2]
The conclusion follows immediately from Corollary A1 when the rounding function is the closest integer function $[\cdot]$.
We next prove the result for the floor function $\lfloor\cdot\rfloor$, noting that the case for the ceiling function $\lceil\cdot\rceil$ is analogous.
Since $\lfloor x\rfloor=[x-1/2]$ and when $n,c\in\N$ we have $\lfloor n^c\rfloor=n^c=[n^c]$, it follows that, for each $=1,\dots,k$, $\lfloor n^{c_i}\rfloor=[f_i(n)]$ where
$$f_i(n)=\begin{cases}
n^{c_i}&\text{ if }c_i\in\N\\
n^{c_i}-1/2&\text{otherwise.}
\end{cases}$$
Since the functions $f_1,\dots,f_k$ satisfy the conditions in Corollary A1, it follows that the conclusion holds.
\end{proof}

Corollary A3 can be shown to follow from Corollary A1.
Alternative, since ${\mathcal K}={\mathcal K}({\mathcal L},q)$ where $q(t)=t$ and ${\mathcal L}=\{t^c\log_m^r:c\in(0,1),m\in\N,r>0\}$, by using \cref{thm_110} one can derive Corollary A3 from part \ref{itm_cond_1} of Theorem~\ref{thm_main_combinatorial}.

The following theorem provides a convenient description of functions which satisfy Condition \ref{itm_cond_1} in the Theorem~\ref{thm_main_combinatorial} and, in particular, implies \cref{thm_110}.


\begin{Theorem}\label{thm_equivalenttoA}
Let $q\in\Z[t]$ be an intersective polynomial, let $\Hardy$ be a Hardy field and let ${\mathcal L}\subset\Hardy$ be a family of functions such that any $f\in{\mathcal L}$ satisfies $t^{k-1}\prec f(t)\prec t^k$ for some $k\in\N$, and any distinct $f,g\in{\mathcal L}$ have different growth.
Let $\tilde {\mathcal K}({\mathcal L},q)$ to be the linear span over $\R$ of ${\mathcal L}$ and $q\R[t]$ and let ${\mathcal K}({\mathcal L},q)$ be the set of functions $f\in\Hardy$ such that there exists $f^*\in\tilde {\mathcal K}({\mathcal L},q)$ with $\lim \big|f(t)-f^*(t)\big|=0$.

Then functions $f_1,\dots,f_k$ from $\Hardy$ satisfy Condition \ref{itm_cond_1} of Theorem~\ref{thm_main_combinatorial} if and only if $\{f_1,\dots,f_k\}\subset{\mathcal K}({\mathcal L},q)$ for some ${\mathcal L}$ and $q$ as above.
\end{Theorem}

\begin{proof}
Suppose that $f_1,\dots,f_k\in{\mathcal K}({\mathcal L},q)$ for some ${\mathcal L}$ and $q$ as in the statement of \cref{thm_equivalenttoA}.
Then $\poly(f_1,\dots,f_k)\subset q\R[t]$, which is spanned by the jointly intersective collection $\{t^\ell q(t):\ell\geq0\}$.
This shows that $f_1,\dots,f_k\in\Hardy$ satisfy Condition \ref{itm_cond_1} of Theorem~\ref{thm_main_combinatorial}.

Conversely, let $\Hardy$ be a Hardy field and let $f_1,\dots,f_k\in\Hardy$ satisfy Condition \ref{itm_cond_1} of Theorem~\ref{thm_main_combinatorial}.
Using \cite[Lemma A.3]{Richter20arXiv} we can find $g_1,\dots,g_m\in\sp(f_1,\dots,f_k)$ of different growth satisfying $t^{\ell-1}\prec g_j(t)\prec t^\ell$ for all $j=1,\dots,m$ and some $\ell\in\N$ (which depends on $j$)
and $p_1,\dots,p_s\in\poly(f_1,\dots,f_k)$ such that for every $i=1,\dots,k$ there exists $f_i^*\in\sp(g_1,\dots,g_m,p_1,\dots,p_s)$ with $\lim|f_i(t)-f_i^*(t)|=0$.
Condition \ref{itm_cond_1} implies that $p_1,\dots,p_s\in\sps(q_1,\dots,q_\ell)$ for some jointly intersective polynomials $q_1,\dots,q_\ell\in\Z[t]$.
In view of \cite[Proposition 6.1]{BLL08} there exists an intersective polynomial $q$ such that $q_1,\dots,q_\ell\in q\Z[t]$, and hence $p_1,\dots,p_s\in q\R[t]$.
Letting ${\mathcal L}=\{g_1,\dots,g_m\}$, we conclude that $f_1,\dots,f_k\in{\mathcal K}({\mathcal L},q)$.

\end{proof}

Finally, it remains to prove Corollary A4.

\begin{proof}[Proof of Corollary A4]
For every $\ell\in\N$ let $F_\ell$ denote the collection of functions
$$F_\ell:=\big\{n\mapsto f_i(n+j):i\in\{1,\dots,k\},j\in\{0,\dots,\ell\}\big\}.$$
In view of Theorem~\ref{thm_main_combinatorial}, it suffices to show that $F_\ell$ satisfies condition \ref{itm_cond_2} in that theorem.
  Considering a finite Taylor expansion of $f_j$, it follows that each function in $F_\ell$ ``almost'' belongs to $\Hsp(f_1,\ldots,f_k)$, in the sense that for any $f\in F_\ell$ there exists $g\in \Hsp(f_1,\ldots,f_k)$ such that $|f(n)-g(n)|\to0$.
  More generally, for any $f\in\sp(F_\ell)$ there exists $g\in \Hsp(f_1,\ldots,f_k)$ such that $|f(n)-g(n)|\to0$.

  Next let $f\in \sp(F_\ell)$ and let $g\in \Hsp(f_1,\ldots,f_k)$ be such that $|f(n)-g(n)|\to0$.
  Suppose, for the sake of a contradiction, that there exists $q\in\Z[t]$ with degree $d$ such that $\lim_{t\to\infty}|f(t)-q(t)|< \infty$.
  Therefore also $\lim_{t\to\infty}|g(t)-q(t)|< \infty$.
  Taking derivatives we have $\lim_{t\to\infty}|g^{(d)}(t)-q^{(d)}(t)|=0$, and since $q^{(d)}$ is a non-zero constant, this contradicts the assumption on $\Hsp(f_1,\dots,f_k)$.
\end{proof}

\section{Characteristic Factors}
\label{sec_charfac}
In this section we show that nilsystems are characteristic for the ergodic averages \eqref{eqn_multiple_recurrence_closest-integer}.
We will make use of the uniformity seminorms introduced in \cite{HK05a} for ergodic systems. As was observed in \cite{CFH11}, the ergodicity of the system is not necessary.

\begin{Definition}
  Let $(X,\mathcal{B},\mu,T)$ be an invertible probability measure preserving system.
  We define the \define{uniformity seminorms} on $L^\infty(X)$ recursively as follows.
  $$\lhk h\rhk_0=\int_Xh\d\mu\qquad\text{and}\qquad\lhk h\rhk_{s}^{2^{s}}=\lim_{N\to\infty}\frac{1}{N}\sum_{n=1}^N\, \lhk \overline{h}\cdot T^nh\rhk_{s-1}^{2^{s-1}}\text{ for every }s\in\N.$$
The existence of the limits in this definition was established in \cite{HK05a} for ergodic systems and in \cite[Section 2.2]{CFH11} in general.
\end{Definition}

Here is the main theorem of this section.

\begin{Theorem}\label{prop_characteristic_factor_main}
Let $\Hardy$ be a Hardy field and $W\in\Hardy$ with $1\prec W(t)\ll t$. Assume $f_1,\ldots,f_k\in \Hardy$ satisfy \ref{property_P_W}, $|f_1(t)|\ll \ldots\ll |f_k(t)|$, and $\lim_{t\to\infty}|f_k(t)|=\lim_{t\to\infty}|f_k(t)-f_i(t)|=\infty$ for every $i<k$. Then there exists $s\in\N$ such that for any invertible measure preserving system $(X,\B,\mu,T)$ and any $h_k\in L^\infty(X)$ with $\lhk h_k\rhk_s=0$ we have
\begin{equation}
\label{eqn_characteristic_factor_main_1}
\sup_{h_1,\ldots,h_{k-1}\in L^\infty}~\sup_{a\in\ell^\infty}~\left\|\frac{1}{W(N)}\sum_{n=1}^N w(n)a(n)\prod_{i=1}^k T^{[f_i(n)]}h_i\right\|_{L^2}= \,\oh_{N\to\infty}(1),
\end{equation}
where the suprema are taken over all functions $h_1,\ldots,h_{k-1}\in L^\infty (X)$ with $\|h_i\|_{L^\infty}\leq 1$ and all $a\in\ell^\infty(\N)$ with $\|a\|_{\ell^\infty}\leq 1$.
\end{Theorem}


\subsection{The sub-linear case of \cref{prop_characteristic_factor_main}}

For the proof of \cref{prop_characteristic_factor_main} we have to distinguish between the case when $f_k$ has at most linear growth (i.e.\ $|f_k(t)|\ll t$), and the case when $f_k$ has super-linear growth (i.e.\ $t\prec |f_k(t)|$), because the respective inductive procedures used to prove these two cases rely on different arguments. 
In this subsection we focus on the proof of the former case, which for the convenience of the reader we state as a separate theorem here.
In this theorem we assume only that the ``weight-function'' $W(t)$ has \define{sub-exponential growth} (i.e., $W(t)\prec c^t$ for all $c>1$), in contrast to \cref{prop_characteristic_factor_main} where $W(t)$ is assumed to have at most linear growth (i.e., $W(t)\ll t$). This is because averages with sub-exponential weights show up naturally in our inductive procedure.

\begin{Theorem}
\label{prop_characteristic_factor_sublinear}
Let $k\in\N$.
Let $\Hardy$ be a Hardy field and $W\in\Hardy$ with $1\prec W(t)$ and suppose $W(t)$ has sub-exponential growth. Assume $f_1,\ldots,f_k\in \Hardy$ satisfy property \ref{property_P_W}, $|f_1(t)|\ll\ldots\ll |f_k(t)|\ll t$, and $\lim_{t\to\infty}|f_k(t)|=\lim_{t\to\infty}|f_k(t)-f_i(t)|=\infty$ for every $i<k$.
Then there exists a constant $C_k>0$, depending only on $k$ and $f_1,\dots,f_k$, such that for any invertible measure preserving system $(X,\B,\mu,T)$ and any $h_k\in L^\infty(X)$ we have
\begin{equation}
\label{eqn_characteristic_factor_sublinear_1}
\sup_{h_1,\ldots,h_{k-1}\in L^\infty}~\sup_{a\in\ell^\infty}~\left\|\frac{1}{W(N)}\sum_{n=1}^N w(n)a(n)\prod_{i=1}^k T^{[f_i(n)]}h_i\right\|_{L^2}\leq\,C_k \lhk h_k\rhk_{k+1}+\oh_{N\to\infty}(1),
\end{equation}
where the suprema are taken over all functions $h_1,\ldots,h_{k-1}\in L^\infty (X)$ with $\|h_i\|_{L^\infty}\leq 1$ and all $a\in\ell^\infty(\N)$ with $\|a\|_{\ell^\infty}\leq 1$.
\end{Theorem}

Given $f\colon\R\to\R$ and $g\colon\R\to(0,\infty)$, we write $f(t)\sim g(t)$ if $\lim_{t\to\infty} f(t)/g(t)=1$.
The first step in proving \cref{prop_characteristic_factor_sublinear} is to reduce it to 
the case $f_k(t)\sim c t$ for some $c>0$ using a method often described as ``change of variables''.

\begin{Lemma}
\label{lem_changeofvariables}
Suppose \cref{prop_characteristic_factor_sublinear} has already been proven for a specific $k\in\N$ when $f_k(t)\sim c t$ for some $c>0$. Then \cref{prop_characteristic_factor_sublinear} follows in full generality for this specific $k$.
\end{Lemma}

\begin{Remark}
\label{rem_subexponential_growth}
It follows from \cite[Lemma 6.2 and Remark 6.3]{Richter20arXiv} that if $W,g\in\Hardy$ with $1\prec W$ and $\log(W(t))\prec g(t)$ then $W\circ g^{-1}$ has sub-exponential growth.
\end{Remark}


\begin{proof}[Proof of \cref{lem_changeofvariables}]
By replacing $f_k$ with $-f_k$ 
and $T$ with its inverse $T^{-1}$ if necessary, we can assume without loss of generality that $f_k(t)$ is eventually positive.
If $f_k(t)\sim ct$ for some $c>0$ then there is nothing to show. Thus, we can assume $f_k(t)\prec t$. Let $g(t)\coloneqq f_k(t)$ and define
$
K_{j}\coloneqq \left\{n\in\N: j < g(n)\leq  j+1\right\}.
$
A straightforward calculation shows that for all but finitely many $j\in\N$ one has
\begin{equation}
\label{eqn_Kj_V}
K_j\,=\,\N\cap \big(g^{-1}(j),g^{-1}(j+1)\big]
\,=\,\big\{\lfloor g^{-1}(j)\rfloor+1,\lfloor g^{-1}(j)\rfloor+2,\ldots,\lfloor g^{-1}(j+1)\rfloor \big\}.
\end{equation}
Since
\begin{equation*}
\begin{split}
\frac{1}{W(N)}\sum_{n=1}^N & w(n) a(n) \prod_{i=1}^k T^{[f_i(n)]}h_i
\\
&=\frac{1}{W(N)}\sum_{j=1}^{\lfloor g(N)\rfloor} \sum_{n\in K_{j}} w(n)a(n)\prod_{i=1}^k T^{[f_i(n)]}h_i\,+\,\oh_{N\to\infty}(1),
\end{split}
\end{equation*}
instead of \eqref{eqn_characteristic_factor_sublinear_1} it suffices to show
\begin{equation}
\label{eqn_characteristic_factor_sublinear_2}
\sup_{h_1,\ldots,h_{k-1}\in L^\infty}~\sup_{a\in\ell^\infty}~\left\|\frac{1}{W(N)}\sum_{j=1}^{\lfloor g(N)\rfloor} \sum_{n\in K_{j}} w(n)a(n)\prod_{i=1}^k T^{[f_i(n)]}h_i\right\|_{L^2}\leq\,C_k \lhk h_k\rhk_{k+1}+\oh_{N\to\infty}(1).
\end{equation}
Set $V=W\circ g^{-1}$. According to \cref{rem_subexponential_growth}, $V(t)$ has sub-exponential growth. In particular, $\lim_{N\to\infty}\sup_{c\in[0,1]}\frac{V(N+c)}{V(N)}=1$, which implies that
\[
\lim_{N\to\infty}\frac{W(N)}{V(\lfloor g(N)\rfloor)}~=~\lim_{N\to\infty}\frac{V(g(N))}{V(\lfloor g(N)\rfloor)}~=~ 1.
\]
Therefore, \eqref{eqn_characteristic_factor_sublinear_2} is equivalent to
\begin{equation*}
\label{eqn_conclusion_of_prop_slow_vdC_3}
\sup_{h_1,\ldots,h_{k-1}\in L^\infty}~\sup_{a\in\ell^\infty}~\left\|\frac{1}{V(\lfloor g(N)\rfloor)}\sum_{j=1}^{\lfloor g(N)\rfloor} \sum_{n\in K_{j}} w(n)a(n)\prod_{i=1}^k T^{[f_i(n)]}h_i\right\|_{L^2}\leq\,C_k \lhk h_k\rhk_{k+1}+\oh_{N\to\infty}(1)
\end{equation*}
which is implied by
\begin{equation}
\label{eqn_conclusion_of_prop_slow_vdC_4}
\sup_{h_1,\ldots,h_{k-1}\in L^\infty}~\sup_{a\in\ell^\infty}~\left\|\frac{1}{V(N)}\sum_{j=1}^{N} \sum_{n\in K_{j}} w(n)a(n)\prod_{i=1}^k T^{[f_i(n)]}h_i\right\|_{L^2}\leq\,C_k \lhk h_k\rhk_{k+1}+\oh_{N\to\infty}(1)
\end{equation}
Define $g_i(t)\coloneqq f_i(g^{-1}(t))$ for $i=1,\ldots,k$ and note that $g_k(t)=t$.
Also note that $g_1,\ldots,g_k,V$ belong to the same Hardy field. Indeed, from $f_1,\ldots,f_k, W\in\Hardy$ we conclude that $g_1,\ldots,g_k,V\in\Hardy\circ g^{-1}$, and according to \cite[Lemma 6.4]{Boshernitzan81} if $g\in\Hardy$ with $\lim_{t\to\infty} g(t)=\infty$ then $\Hardy\circ g^{-1}=\{f\circ g^{-1}: f\in\Hardy\}$ is a Hardy field.
Since $|f_i(t)|\ll f_k(t)$ for all $i\leq k$, there exists $C\in\N$ such that $|f_i(t)|< C f_k(t)$ for all $t$ in some half-line $[t_0,\infty)$.
This implies that $|f_i(n)-g_i(j)|< C$ for all $n\in K_j$ and all but finitely many $j\in\N$.
In other words, for all but finitely many $j\in\N$ and all $n\in K_j$ we have $[f_i(n)]\in [g_i(j)]+\{-C,\ldots,C\}$. For every $\eta=(\eta_1,\ldots,\eta_k)\in\{-C,\ldots,C\}^{k}$ define
\begin{eqnarray*}
K_{j}^{(\eta)}&\coloneqq & \{n\in K_j: [f_1(n)]= [g_1(j)]+\eta_1,\ldots, [f_k(n)]= [g_k(j)]+\eta_k\}.
\end{eqnarray*}
Then, $\bigcup_{\eta\in \{-C,\ldots,C\}^k} K_{j}^{(\eta)}=K_j$. Therefore, to prove \eqref{eqn_conclusion_of_prop_slow_vdC_4}, it suffices to show that for every $\eta=(\eta_1,\ldots,\eta_k)\in\{-C,\ldots,C\}^{k}$ we have
\begin{equation}
\label{eqn_conclusion_of_prop_slow_vdC_5}
\sup_{h_1,\ldots,h_{k-1}\in L^\infty}\,\sup_{a\in\ell^\infty}\left\|\frac{1}{V(N)}\sum_{j=1}^{N} \sum_{n\in K_{j}^{(\eta)}} w(n)a(n)\prod_{i=1}^k T^{[g_i(j)]}T^{\eta_i}h_i\right\|_{L^2}\leq\,C_{k,\eta} \lhk h_k\rhk_{k+1}+\oh_{N\to\infty}(1),
\end{equation}
for some constants $C_{k,\eta}$.
Note that $\lhk h_k\rhk_{k+1}=\lhk T^{\eta_k}h_{k}\rhk_{k+1}$. Moreover, the supremum over all functions $h_1,\ldots,h_{k-1}\in L^\infty (X)$ with $\|h_i\|_{L^\infty}\leq 1$ is the same as the supremum over all functions $T^{\eta_1}h_{1},\ldots,T^{\eta_{k-1}}h_{k-1}\in L^\infty (X)$ with $\|T^{\eta_i}h_{i}\|_{L^\infty}\leq 1$. This means that \eqref{eqn_conclusion_of_prop_slow_vdC_5} holds if and only if for all $h_k\in L^\infty(X)$ 
we have
\begin{equation}
\label{eqn_conclusion_of_prop_slow_vdC_6}
\sup_{h_1,\ldots,h_{k-1}\in L^\infty}~\sup_{a\in\ell^\infty}~\left\|\frac{1}{V(N)}\sum_{j=1}^{N} \sum_{n\in K_{j}^{(\eta)}} w(n)a(n)\prod_{i=1}^k T^{[g_i(j)]}h_i\right\|_{L^2}\leq\,C_{k,\eta} \lhk h_k\rhk_{k+1}+\oh_{N\to\infty}(1),
\end{equation}
where the suprema are taken over all functions $h_1,\ldots,h_{k-1}\in L^\infty (X)$ with $\|h_i\|_{L^\infty}\leq 1$ and all $a\in\ell^\infty(\N)$ with $\|a\|_{\ell^\infty}\leq 1$.
Define $b\in\ell^\infty(\N)$ as
$$
b(j)\,\coloneqq \, \frac{\sum_{n\in K_{j}^{(\eta)}} w(n)a(n)}{\sum_{n\in K_{j}} w(n)},\qquad\forall j\in\N,
$$
and note that $\|b\|_{\ell^\infty}\leq 1$ since $\|a\|_{\ell^\infty}\leq 1$. Thus \eqref{eqn_conclusion_of_prop_slow_vdC_6} becomes
 \begin{equation}
\label{eqn_conclusion_of_prop_slow_vdC_7}
\sup_{h_1,\ldots,h_{k-1}\in L^\infty}~\sup_{b\in\ell^\infty}~\left\|\frac{1}{V(N)}\sum_{j=1}^{N}b(j)\left(\sum_{n\in K_{j}} w(n)\right)\prod_{i=1}^k T^{[g_i(j)]}h_{i}\right\|_{L^2}\leq\,C_{k,\eta} \lhk h_k\rhk_{k+1}+\oh_{N\to\infty}(1).
\end{equation}
Let $v(n)\coloneqq \Delta V(n)=V(n+1)-V(n)$.
Pick $y_j\in\R$ such that $\lfloor g^{-1}(j)\rfloor+1=g^{-1}(j+y_j)$ and observe that $y_j\to0$ as $j\to\infty$ because $g(t)=f_k(t)\prec t$. Then
\begin{align*}
\frac{W(\lfloor g^{-1}(j)\rfloor +1)}{v(j)}=\frac{V(j+y_j)}{v(j)}
=\frac{V(j)}{v(j)}+\frac{V(j+y_j)-V(j)}{V(j+1)-V(j)}
\end{align*}
A straightforward application of the mean value theorem shows that $\frac{V(j+y_j)-V(j)}{V(j+1)-V(j)}=\Oh(y_j)$ and hence
\begin{align}
\label{eqn_quot_V_v_1}
\frac{W(\lfloor g^{-1}(j)\rfloor +1)}{v(j)}=\frac{V(j)}{v(j)}+\oh_{j\to\infty}(1).
\end{align}
A similar calculation shows
\begin{align}
\label{eqn_quot_V_v_2}
\frac{W(\lfloor g^{-1}(j+1)\rfloor +1)}{v(j)}=
\frac{V(j+1)}{v(j)}+\oh_{j\to\infty}(1).
\end{align}
By invoking \eqref{eqn_Kj_V} and combining \eqref{eqn_quot_V_v_1} and \eqref{eqn_quot_V_v_2} we get
\begin{align*}
\frac{\sum_{n\in K_{j}} w(n)}{v(j)}&=\frac{W(\lfloor g^{-1}(j+1)\rfloor +1)-W(\lfloor g^{-1}(j)\rfloor+1)}{v(j)}
\\
&=\frac{V(j+1)-V(j)}{v(j)}+\oh_{j\to\infty}(1)
\\
&=1+\oh_{j\to\infty}(1).
\end{align*}
Therefore \eqref{eqn_conclusion_of_prop_slow_vdC_7} is equivalent to
\begin{equation}
\label{eqn_conclusion_of_prop_slow_vdC_8_ne}
\sup_{h_1,\ldots,h_{k-1}\in L^\infty}~\sup_{b\in\ell^\infty}~\left\|\frac{1}{V(N)}\sum_{j=1}^{N}b(j)v(j)\prod_{i=1}^k T^{[g_i(j)]}h_{i}\right\|_{L^2}\leq\,C_{k,\eta} \lhk h_k\rhk_{k+1}+\oh_{N\to\infty}(1).
\end{equation}
Since $g_k(t)=t$, \eqref{eqn_conclusion_of_prop_slow_vdC_8_ne} follows from the hypothesis of the lemma once we have verified that $g_1,\ldots,g_k$ satisfy property~\ref{property_P_W} with $W(t)$ replaced by $V(t)$.
In other words, it remains to show that for all $f\in\Hsp(g_1,\ldots,g_k)$ and $p\in\R[t]$ either $|f(t)-p(t)|\ll 1$ or $|f(t)-p(t)| \succ \log(V(t))$.
Since  $g_1,\ldots,g_k$ have at most linear growth, all their derivatives are either asymptotically constant, or asymptotically negligible. Therefore, it suffices to show the property for $f\in\sps(g_1,\ldots,g_k)$ instead of $f\in\Hsp(g_1,\ldots,g_k)$.
This follows readily from the fact that $f_1,\ldots,f_k$ satisfy property~\ref{property_P_W} and $g_i=f_i\circ g^{-1}$ and $V=W\circ g^{-1}$.
\end{proof}

For the inductive step in the proof of \cref{prop_characteristic_factor_sublinear} we need the following version of van der Corput's Lemma.

\begin{Lemma}
\label{prop_vdC}
Let $p_1, p_2,p_3,\ldots$ be an eventually monotone sequence of positive real numbers.
Let $P_N\coloneqq \sum_{n=1}^N p_n$ and assume
$$
\lim_{N\to\infty} P_N\, =\, \infty
\qquad\text{and}\qquad
\lim_{N\to\infty}\frac{p_N}{P_N} \, =\, 0.
$$
Let $\Hilb$ be a Hilbert space and, for every $N\in\N$, let $u_N\colon\N\to\Hilb$ be a sequence bounded in norm by $1$.
Then
\begin{equation}
\label{eqn_prop_vdC_1}
\begin{split}
\limsup_{N\to\infty}\Bigg\|\frac1{P_N}\sum_{n=1}^N & p_nu_N(n)\Bigg\|^2
\\
&\,\leq\,\limsup_{H\to\infty}\limsup_{N\to\infty}\left|\frac{1}{H}\sum_{m=1}^H \frac1{P_N}\sum_{n=1}^N p_n\big\langle u_N(n+m),u_N(n)\big\rangle\right|.
\end{split}
\end{equation}
\end{Lemma}

\begin{proof}
Pick $N_1<N_2<N_3<\ldots \in\N$ such that
$$
\limsup_{N\to\infty}\Bigg\|\frac1{P_N}\sum_{n=1}^N  p_nu_N(n)\Bigg\|~=~\lim_{k\to\infty}\Bigg\|\frac1{P_{N_k}}\sum_{n=1}^{N_k} p_n u_{N_k}(n)\Bigg\|.
$$
By further refining the subsequence $(N_k)_{k\in\N}$ if necessary, we can also assume that for all $m\in\N$ the limit
$$
\lim_{k\to\infty}\frac1{P_{N_k}}\sum_{n=1}^{N_k} p_n\big\langle u_{N_k}(n+m),u_{N_k}(n)\big\rangle
$$
exists. Then \eqref{eqn_prop_vdC_1} follows if we can show
\begin{equation}
\label{eqn_prop_vdC_1_2}
\begin{split}
\lim_{k\to\infty}\Bigg\|\frac1{P_{N_k}}\sum_{n=1}^{N_k} & p_n u_{N_k}(n)\Bigg\|^2
\\
&\,\leq\,\limsup_{H\to\infty}\lim_{k\to\infty}\left|\frac{1}{H}\sum_{m=1}^H \frac1{P_{N_k}}\sum_{n=1}^{N_k} p_n\big\langle u_{N_k}(n+m),u_{N_k}(n)\big\rangle\right|.
\end{split}
\end{equation}

First, observe that
\begin{eqnarray*}
\left\|\frac1{P_{N_k}}\sum_{n=1}^{N_k}p_nu_{N_k}(n)-\frac1{P_{N_k}}\sum_{n=1}^{N_k}p_nu_{N_k}(n+1)\right\|
&=&
\left\|\frac1{P_{N_k}}\sum_{n=1}^{N_k}p_nu_{N_k}(n)-\sum_{n=2}^{{N_k}+1}p_{n-1}u_{N_k}(n)\right\|
\\&\leq&
\frac 1{P_{N_k}}\sum_{n=2}^{{N_k}}|p_n-p_{n-1}|+\oh_{k\to\infty}(1)\end{eqnarray*}
Since $p_n$ is monotonic, the latter quantity can be bounded by
$$\frac 1{P_{N_k}}|p_{N_k}|+\oh_{k\to\infty}(1)=\oh_{k\to\infty}(1)$$
Iterating this observation and (\Cesaro{}) averaging we deduce that, for any $H\in\N$,
$$\lim_{k\to\infty}\left\|\frac1{P_{N_k}}\sum_{n=1}^{N_k}p_nu_{N_k}(n)-\frac{1}{H}\sum_{m=1}^H\frac1{P_{N_k}}\sum_{n=1}^{N_k}p_nu_{N_k}(n+m)\right\|=0.$$
In particular,
\begin{equation}
\label{eqn_prop_vdC_2.7}
\lim_{k\to\infty}\Bigg\|\frac1{P_{N_k}}\sum_{n=1}^{N_k}p_nu_{N_k}(n)\Bigg\|^2 \,=\,
\lim_{k\to\infty}\Bigg\|\frac{1}{H}\sum_{m=1}^H  \frac1{P_{N_k}}\sum_{n=1}^{N_k}p_nu_{N_k}(n+m)\Bigg\|^2.
\end{equation}
Using Jensen's inequality and expanding the square on the right hand side of \eqref{eqn_prop_vdC_2.7} leaves us with
\begin{equation*}
\begin{split}
\lim_{k\to\infty} & \Bigg\| \frac{1}{H} \sum_{m=1}^H \frac1{P_{N_k}}\sum_{n=1}^{N_k}p_nu_{N_k}(n+m)\Bigg\|^2
  \\&\leq
\lim_{k\to\infty}\frac1{P_{N_k}}\sum_{n=1}^{N_k}p_n\left\|\frac{1}{H}\sum_{m=1}^Hu_{N_k}(n+m)\right\|^2
  \\&=
\frac{1}{H^2}\sum_{m_1,m_2=1}^H \left(\lim_{k\to\infty}\frac1{P_{N_k}}\sum_{n=1}^{N_k}p_n\big\langle u_{N_k}(n+m_1),u_{N_k}(n+m_2)\big\rangle \right)
\\&=
\frac{1}{H^2}\hspace{-.3 em}\sum_{1\leq m_2<m_1\leq H}\hspace{-.3 em}2~\Re\left(\lim_{k\to\infty}\frac1{P_{N_k}}\sum_{n=1}^{N_k}p_n\big\langle u_{N_k}(n+m_1),u_{N_k}(n+m_2)\big\rangle\right)\,+\,\Oh\left(\frac{1}{H}\right)
\\&=
\frac{1}{H^2}\hspace{-.3 em}\sum_{1\leq m_2<m_1\leq H}\hspace{-.3 em}2~\Re\left(\lim_{k\to\infty}\frac1{P_{N_k}}\sum_{n=1}^{N_k}p_n\big\langle u_{N_k}(n+m_1-m_2),u_{N_k}(n)\big\rangle\right)\,+\,\Oh\left(\frac{1}{H}\right)
\\&=
\Re\left(\frac{1}{H}\sum_{m=1}^H \tfrac{2(H-m)}{H}\left(\lim_{k\to\infty}\frac1{P_{N_k}}\sum_{n=1}^{N_k}p_n\big\langle u_{N_k}(n+m),u_{N_k}(n)\big\rangle\right)\right)\,+\,\Oh\left(\frac{1}{H}\right)
\\&\leq
\left|\frac{1}{H}\sum_{m=1}^H \tfrac{2(H-m)}{H}\left(\lim_{k\to\infty}\frac1{P_{N_k}}\sum_{n=1}^{N_k}p_n\big\langle u_{N_k}(n+m),u_{N_k}(n)\big\rangle\right)\right|\,+\,\Oh\left(\frac{1}{H}\right)
\end{split}
\end{equation*}
Since the function
$$
\Psi(m)\,\coloneqq\,\lim_{k\to\infty}\frac1{P_{N_k}}\sum_{n=1}^{N_k}p_n\big\langle u_{N_k}(n+m),u_{N_k}(n)\big\rangle
$$
is positive definite, its \define{uniform \Cesaro{} average} exists, meaning that for any $(L_k)_{k\in\N},(M_k)_{k\in\N}\subset\N$ with $M_k-L_k\to\infty$ as $k\to\infty$ the limit
$$
\lim_{k\to\infty} \frac{1}{M_k-L_k}\sum_{m=L_k}^{M_k-1} \Psi(m)
$$
exists and equals
$$
\lim_{H\to\infty}\frac{1}{H}\sum_{m=1}^H\Psi(m).
$$
In particular, this means that
$$
\lim_{H\to\infty}\frac{1}{H}\sum_{m=1}^H \tfrac{2(H-m)}{H}\Psi(m)\,=\,\lim_{H\to\infty}\frac{1}{H}\sum_{m=1}^H\Psi(m).
$$
This shows that
\begin{equation*}
\begin{split}
\lim_{k\to\infty} \Bigg\| \frac{1}{H} \sum_{m=1}^H & \frac1{P_{N_k}}\sum_{n=1}^{N_k}p_nu_{N_k}(n+m)\Bigg\|^2
\\
&\,\leq\,\limsup_{H\to\infty}\limsup_{N\to\infty}\left|\frac{1}{H}\sum_{m=1}^H \frac1{P_N}\sum_{n=1}^N p_n\big\langle u_N(n+m),u_N(n)\big\rangle\right|,
\end{split}
\end{equation*}
which finishes the proof.
\end{proof}


\begin{proof}[Proof of \cref{prop_characteristic_factor_sublinear}]
We use induction on $k$. For the base case of this induction, which corresponds to $k=1$, we have to show that for any invertible measure preserving system $(X,\B,\mu,T)$, any $h_1\in L^\infty(X)$, and any function $f_1\in\Hardy$ with $|f_1(t)|\ll t$, we have
\begin{equation}
\label{eqn_characteristic_factor_sublinear_71}
\sup_{a\in\ell^\infty}~\left\|\frac{1}{W(N)}\sum_{n=1}^N w(n) a(n)T^{[f_1(n)]}h_1\right\|_{L^2}\leq\,C_1 \lhk h_1\rhk_2+\oh_{N\to\infty}(1),
\end{equation}
where the supremum is taken over all $a\in\ell^\infty(\N)$ with $\|a\|_{\ell^\infty}\leq 1$. In light of \cref{lem_changeofvariables}, we can assume without loss of generality that $f_1(t)\sim ct$ for some $c>0$.
Instead of taking the supremum in \eqref{eqn_characteristic_factor_sublinear_71}, it suffices to show that for any $a_1,a_2,\ldots\in\ell^\infty(\N)$ with $\|a_N\|_{\ell^\infty}\leq 1$ we have
\begin{equation}
\label{eqn_characteristic_factor_sublinear_71new}
\left\|\frac{1}{W(N)}\sum_{n=1}^N w(n) a_N(n)T^{[f_1(n)]}h_1\right\|_{L^2}\leq\,C_1 \lhk h_1\rhk_2+\oh_{N\to\infty}(1).
\end{equation}
Note that $\lim_{N\to\infty}w(N)/W(N)=0$ because $W(t)$ has sub-exponential growth. By applying \cref{prop_vdC} with $W(N)=P_N$, $w(n)=p_n$, and $u_N(n)=a_N(n)T^nh_1$, we see that \eqref{eqn_characteristic_factor_sublinear_71new} follows if we can show
\begin{equation}
\label{eqn_characteristic_factor_sublinear_72}
\begin{split}
\limsup_{H\to\infty}\,\limsup_{N\to\infty}~\frac{1}{H} \sum_{m=1}^H\Bigg| \frac{1}{W(N)}\sum_{n=1}^N & w(n)\, a_N(n+m)\overline{a}_N(n)
\\
&\int T^{[f_1(n+m)]-[f_1(n)]}h_1\cdot\overline{h}_1\d\mu\Bigg|\leq\,C_1 \lhk h_1\rhk_2^2.
\end{split}
\end{equation}
Since $f_1(t)\sim ct$, we have $[f_1(n+m)]-[f_1(n)]\in[cm]+\{-2,-1,0,1,2\}$ for all but finitely many $n\in\N$.
Define, for $\eta\in\{-2,-1,0,1,2\}$, the set $V_\eta\coloneqq\{n\in\N: [f_1(n+m)]-[f_1(n)]=[cm]+\eta\}$. Then, since $\bigcup_{\eta\in\{-2,-1,0,1,2\}}V_\eta$ is co-finite in $\N$, in place of \eqref{eqn_characteristic_factor_sublinear_72} it suffices to show that for all $\eta\in\{-2,-1,0,1,2\}$,
\begin{equation}
\label{eqn_characteristic_factor_sublinear_72_2}
\begin{split}
\limsup_{H\to\infty}\,\limsup_{N\to\infty}~\frac{1}{H} \sum_{m=1}^H\Bigg| \frac{1}{W(N)}\sum_{n=1}^N w(n)\, a_N(n+m) & \overline{a}_N(n)1_{V_\eta}(n)
\\
&\int T^{[cm]+\eta}h_1\cdot\overline{h}_1\d\mu\Bigg|\leq\,\frac{C_1}{5}\lhk h_1\rhk_2^2.
\end{split}
\end{equation}
We can rewrite \eqref{eqn_characteristic_factor_sublinear_72_2} as
\begin{equation*}
\begin{split}
\limsup_{H\to\infty}\,\limsup_{N\to\infty}~\frac{1}{H}\sum_{m=1}^H  \Bigg|\int T^{[cm]+\eta}h_1\cdot & \overline{h}_1 \d\mu \Bigg|
\\
&\left|\frac{1}{W(N)}\sum_{n=1}^N w(n) a_N(n+m)\overline{a}_N(n)1_{V_\eta}(n)\right|\leq\, \frac{C_1}{5}\lhk h_1\rhk_2^2,
\end{split}
\end{equation*}
which is implied by
\begin{equation}
\label{eqn_characteristic_factor_sublinear_73}
\lim_{H\to\infty}~\frac{1}{H}\sum_{m=1}^H \left|\int T^{[cm]+\eta}h_1\cdot\overline{h}_1   \d\mu\right|\leq\,\frac{C_1}{5}\lhk h_1\rhk_2^2.
\end{equation}
To see why \eqref{eqn_characteristic_factor_sublinear_73} holds, let $r(m)\coloneqq |\{n\in\N: [cn]+\eta=m\}|$ and observe that
\begin{equation*}
\lim_{H\to\infty}\frac{1}{H}\sum_{m=1}^H \left|\int T^{[cm]+\eta}h_1\cdot\overline{h}_1 \d\mu\right|
~=~
\lim_{H\to\infty}\frac{c}{H}\sum_{m=1}^H r(m) \left|\int T^{m}h_1\cdot\overline{h}_1   \d\mu\right|.
\end{equation*}
Since $r(m)\leq 1+c^{-1}$, we have
\begin{equation*}
\lim_{H\to\infty}\frac{c}{H}\sum_{m=1}^H r(m) \left|\int T^{m}h_1\cdot\overline{h}_1   \d\mu\right|
~\leq~
(c+1)\left(\lim_{H\to\infty}\frac{1}{H}\sum_{m=1}^H \left|\int T^{m}h_1\cdot\overline{h}_1\d\mu\right|\right).
\end{equation*}
But
\begin{equation*}
\lim_{H\to\infty}\frac{1}{H}\sum_{m=1}^H \left|\int T^{m}h_1\cdot\overline{h}_1\d\mu\right|
~\leq~\lhk h_1\rhk_2^2,
\end{equation*}
by the definition of $\lhk .\rhk_2$, and so \eqref{eqn_characteristic_factor_sublinear_73} holds with $C_1=5(c+1)$.

For the proof of the inductive step, assume \cref{prop_characteristic_factor_sublinear} has already been proven for $k-1$; we want to show
\begin{equation}
\label{eqn_characteristic_factor_sublinear_74}
\sup_{h_1,\ldots,h_{k-1}}~\sup_{a\in\ell^\infty}~\left\|\frac{1}{W(N)}\sum_{n=1}^N w(n)a(n)\prod_{i=1}^k T^{[f_i(n)]}h_i\right\|_{L^2}\leq\,C_k\lhk h_k\rhk_{k+1}+\oh_{N\to\infty}(1).
\end{equation}
Again, in light of \cref{lem_changeofvariables} we can assume $f_k(t)\sim ct$ for some $c>0$, and instead of having the suprema $\sup_{h_1,\ldots,h_{k-1}}$ and $\sup_{a\in\ell^\infty}$ in \eqref{eqn_characteristic_factor_sublinear_74}, it suffices to show that for any $h_{1,N},\ldots,h_{k-1,N}\in L^\infty (X)$ with $\|h_{i,N}\|_{L^\infty}\leq 1$ and any $a_N\in\ell^\infty(\N)$ with $\|a_N\|_{\ell^\infty}\leq 1$ we have
\begin{equation}
\label{eqn_characteristic_factor_sublinear_75}
\left\|\frac{1}{W(N)}\sum_{n=1}^N w(n) a_N(n)\prod_{i=1}^k T^{[f_i(n)]}h_{i,N}\right\|_{L^2}\leq\,C_k\lhk h_k\rhk_{k+1}+\oh_{N\to\infty}(1),
\end{equation}
where, for convenience, we took $h_{k,N}=h_k$ for all $N\in\N$.
We apply \cref{prop_vdC} once more, this time with $u_N(n)=a_N(n)\prod_{i=1}^k T^{[f_i(n)]}h_{i,N}$, and deduce that \eqref{eqn_characteristic_factor_sublinear_75} holds if
\begin{equation}
\label{eqn_characteristic_factor_sublinear_76}
\begin{split}
\limsup_{H\to\infty}\limsup_{N\to\infty}\frac{1}{H}\sum_{m=1}^H\Bigg|\frac{1}{W(N)} & \sum_{n=1}^N  w(n) a_N(n+m)\overline{a}_N(n)
\\
&\int \prod_{i=1}^k T^{[f_i(n+m)]}h_{i,N}\cdot T^{[f_i(n)]}\overline{h}_{i,N}  \d\mu\Bigg|\leq\, C_k\lhk h_k\rhk_{k+1}^2.
\end{split}
\end{equation}
Next, we write \eqref{eqn_characteristic_factor_sublinear_76} as
\begin{equation}
\label{eqn_characteristic_factor_sublinear_77}
\begin{split}
\limsup_{H\to\infty}\limsup_{N\to\infty}~ & \frac{1}{H}\sum_{m=1}^H\Bigg| \frac{1}{W(N)}  \sum_{n=1}^N w(n) a_N(n+m)\overline{a}_N(n)
\\
&\int \prod_{i=1}^k T^{[f_i(n)]-[f_1(n)]}\Big(T^{[f_i(n+m)]-[f_i(n)]}h_{i,N}\cdot\overline{h}_{i,N}\Big) \d\mu\Bigg|\leq\, C_k\lhk h_k\rhk_{k+1}^2.
\end{split}
\end{equation}
Define $c_i\coloneqq \lim_{t\to\infty}f_i(t)/t$. Then, arguing as above, we have
$$
[f_i(n+m)]-[f_i(n)]\,\in\, [c_i m]+\{-2,-1,0,1,2\}.
$$
Moreover,
$$
[f_i(n)]-[f_1(n)] \, \in\,  [f_i(n)-f_1(n)]+\{-1,0,1\}.
$$
Thus, if we define for $\eta=(\eta_1,\ldots,\eta_k)\in\{-2,-1,0,1,2\}^{k}$ and $\kappa=(\kappa_1,\ldots,\kappa_k)\in\{-1,0,1\}$ the set
$$
V_{\eta,\kappa}\coloneqq\{n\in\N: [f_i(n+m)]-[f_i(n)]=[c_i m]+\eta_i,~[f_i(n)]-[f_1(n)]=[f_i(n)-f_1(n)]+\kappa_i\},
$$
then the union $\bigcup_{\eta\in\{-2,-1,0,1,2\}^{k}}\bigcup_{\kappa\in\{-1,0,1\}^{k}} V_{\eta,\kappa}$ is co-finite in $\N$.
Therefore, to establish \eqref{eqn_characteristic_factor_sublinear_77}, it suffices to show that for every $\eta\in\{-2,-1,0,1,2\}^{k}$ and $\kappa\in\{-1,0,1\}^{k}$ we have
\begin{equation}
\label{eqn_characteristic_factor_sublinear_78}
\begin{split}
\limsup_{H\to\infty}\limsup_{N\to\infty}~ & \frac{1}{H}\sum_{m=1}^H\Bigg| \frac{1}{W(N)}  \sum_{n=1}^N w(n) a_N(n+m)\overline{a}_N(n) 1_{V_{\eta,\kappa}}(n)
\\
&\int \prod_{i=1}^k T^{[f_i(n)-f_1(n)]+\kappa_i}\Big(T^{[c_im]+\eta_i}h_{i,N}\cdot\overline{h}_{i,N}\Big) \d\mu\Bigg|\leq\, \frac{C_k}{5^k3^k}\lhk h_k\rhk_{k+1}^2.
\end{split}
\end{equation}
Set
$$
\tilde{f}_{i-1}(t)\coloneqq f_{i}(t)-f_1(t)\quad\text{and}\quad
\tilde{h}_{i-1,m,N}\coloneqq T^{\kappa_i}\Big(T^{[c_im]+\eta_i}h_{i,N}\cdot\overline{h}_{i,N}\Big),~ i=2,\ldots,k,
$$
and take $\tilde{a}_{m,N}(n)\coloneqq a_N(n+m)\overline{a}_N(n) 1_{V_{\eta,\kappa}}(n)$. Then \eqref{eqn_characteristic_factor_sublinear_78} is implied by
\begin{equation}
\label{eqn_characteristic_factor_sublinear_79}
\limsup_{H\to\infty}\limsup_{N\to\infty} \frac{1}{H}\sum_{m=1}^H\left\|\frac{1}{W(N)}\sum_{n=1}^N w(n) \tilde{a}_{m,N}(n)\prod_{i=1}^{k-1} T^{[\tilde{f}_i(n)]}\tilde{h}_{i,m,N}\right\|_{L^2}\leq\, \frac{C_k}{5^k3^k}\lhk h_k\rhk_{k+1}^2.
\end{equation}
By the induction hypothesis, for each $m\in\N$ we have the estimate
$$
\limsup_{N\to\infty}\left\|\frac{1}{W(N)}\sum_{n=1}^N w(n) \tilde{a}_{m,N}(n)\prod_{i=1}^{k-1} T^{[\tilde{f}_i(n)]}\tilde{h}_{i,m,N}\right\|_{L^2}\leq\, C_{k-1}\lhk T^{[c_km]+\eta_k}h_k\cdot\overline{h}_k\rhk_{k}.
$$
Hence, \eqref{eqn_characteristic_factor_sublinear_79} follows from
$$
\limsup_{H\to\infty}\frac{1}{H}\sum_{m=1}^H C_{k-1}\lhk T^{[c_km]+\eta_k}h_k\cdot\overline{h}_k\rhk_{k}\leq\, \frac{C_k}{5^k3^k}\lhk h_k\rhk_{k+1}^2,
$$
which holds for $C_k\coloneqq(c_k+1) 5^k3^k C_{k-1}$ and can be proved similarly to \eqref{eqn_characteristic_factor_sublinear_73}.
\end{proof}

\subsection{The general case of \cref{prop_characteristic_factor_main}}

The induction that we will use to prove \cref{prop_characteristic_factor_main} is similar to the PET-induction scheme utilized in \cite{Bergelson87}.
Given a function $f\in\Hardy$ of polynomial growth, we call the smallest $d\in\N$ for which $|f(t)|\ll t^d$ the \define{degree} of $f$, and denote it by $\deg(f)$.
Consider the relation $f\sim_{\operatorname{PET}} g$ if and only if $\deg f=\deg g>\deg(f-g)$ and note that $\sim_{\operatorname{PET}}$ defines an equivalence relation on $\Hardy$.

Now given a finite collection $\F=\{f_1,\ldots,f_k\}$ of functions of polynomial growth from a Hardy field $\Hardy$, let $d_{\max}$ be a number that is bigger or equal than the degree of any function in $\F$.
Also, for each $d\in\{1,\dots,d_{\max}\}$, let $m_d$ denote the number of equivalence classes of the set $\{f\in\F: \deg(f)=d\}$ with respect to the equivalence relation $\sim_{\operatorname{PET}}$.
The vector $(m_1,\dots,m_{d_{\max}})$ is called the \define{characteristic vector} of $\F$.
We order characteristic vectors by letting $(m_1,\dots,m_{d_{\max}})<(\tilde m_1,\dots,\tilde m_{d_{\max}})$ if the maximum $j$ for which $m_j\neq\tilde m_j$ satisfies $m_j<\tilde m_j$.

\begin{proof}[Proof of {\cref{prop_characteristic_factor_main}}]
We prove the theorem by induction on the characteristic vector of $\F=\{f_1,\dots,f_k\}$.
The case when $d_{\max}=1$ (i.e., all functions in $\F$ have degree $1$) follows from \cref{prop_characteristic_factor_sublinear}. We can therefore assume that $f_k$ has degree at least $2$ and that the theorem has been proved for all families whose characteristic vector is strictly smaller than that of $\F$.

After reordering, if necessary, we fall into one of the following two cases: either all functions in $\F$ are equivalent or $f_1\not\sim_{\operatorname{PET}} f_k$ (while keeping $|f_1(t)|\ll \ldots \ll |f_k(t)|$).

Our goal is to show
\begin{equation*}
\limsup_{N\to\infty}~\sup_{h_1,\ldots,h_{k-1}\in L^\infty}~\sup_{a\in\ell^\infty}~\left\|\frac{1}{W(N)}\sum_{n=1}^N w(n)a(n)\prod_{i=1}^k T^{[f_i(n)]}h_i\right\|_{L^2}= \,0,
\end{equation*}
where the suprema are taken over all functions $h_1,\ldots,h_{k-1}\in L^\infty (X)$ with $\|h_i\|_{L^\infty}\leq 1$ and all $a\in\ell^\infty(\N)$ with $\|a\|_{\ell^\infty}\leq 1$. As we did in the proof of \cref{prop_characteristic_factor_sublinear} in the previous subsection, instead of asking for the suprema over $h_1,\ldots,h_{k-1}$ and $a$, it suffices to show that for any $h_{1,N},\ldots,h_{k-1,N}\in L^\infty (X)$ with $\|h_{i,N}\|_{L^\infty}\leq 1$ and any $a_N\in\ell^\infty(\N)$ with $\|a\|_{\ell^\infty}\leq 1$ we have
\begin{equation}
\label{eqn_characteristic_factor_main_1-1}
\limsup_{N\to\infty}~\left\|\frac{1}{W(N)}\sum_{n=1}^N w(n)a_N(n)\prod_{i=1}^k T^{[f_i(n)]}h_{i,N}\right\|_{L^2}= \,0,
\end{equation}
where, for convenience, we took $h_{k,N}=h_k$ for all $N\in\N$.
Then we use \cref{prop_vdC} with $u_N(n)=a_N(n)\prod_{i=0}^k T^{[f_i(n)]}h_{i,N}$ and conclude that in order to establish \eqref{eqn_characteristic_factor_main_1-1} it suffices to show that for every $m\in\N$
\begin{equation}
\label{eqn_characteristic_factor_main_1-2}
\begin{split}
\lim_{N\to\infty}\Bigg|\frac{1}{W(N)} & \sum_{n=1}^N  w(n) a_N(n+m)\overline{a}_N(n)
\\
&\int \prod_{i=1}^k T^{[f_i(n+m)]}h_{i,N}\cdot T^{[f_i(n)]}\overline{h}_{i,N}  \d\mu\Bigg|=\, 0.
\end{split}
\end{equation}
Note that $[f_i(n+m)]-[f_{1}(n)]=[f_i(n+m)- f_{1}(n)]+e_{i,n}$ where $e_{i,n}\in\{-1,0,1\}$ and $[f_i(n)]-[f_{1}(n)]=[f_i(n)- f_{1}(n)]+\tilde e_{i,n}$ where $\tilde e_{i,n}\in\{-1,0,1\}$.
For each vector $v=(v_1,\dots,v_k,\tilde v_1,\dots\tilde v_k)\in\{-1,0,1\}^{2k}$, let $A_v$ be the set of $n$'s for which
$$
(e_{1,n},\dots,e_{k,n},\tilde e_{1,n},\dots\tilde e_{k,n})=v.
$$
Since $\N=\bigcup_{v\in\{-1,0,1\}^{2k}} A_v$, instead of \eqref{eqn_characteristic_factor_main_1-2} it suffices to show that for every $v\in\{-1,0,1\}^{2k}$ we have
\begin{equation}
\label{eqn_characteristic_factor_main_1-3}
\begin{split}
\lim_{N\to\infty}\Bigg|\frac{1}{W(N)} & \sum_{n=1}^N  w(n) a_N(n+m)\overline{a}_N(n) 1_{A_v}(n)
\\
&\int \prod_{i=1}^k T^{[f_i(n+m)]}h_{i,N}\cdot T^{[f_i(n)]}\overline{h}_{i,N}  \d\mu\Bigg|=\, 0.
\end{split}
\end{equation}
But for $n\in A_v$ we have
\begin{equation*}
\begin{split}
\int \prod_{i=1}^k T^{[f_i(n+m)]} & h_{i,N}\cdot T^{[f_i(n)]}\overline{h}_{i,N}  \d\mu
\\
&=\int \prod_{i=1}^k T^{[f_i(n+m)]-[f_1(n)]}h_{i,N}\cdot T^{[f_i(n)]-[f_1(n)]}\overline{h}_{i,N}  \d\mu
\\
&=\int \prod_{i=1}^k T^{[f_i(n+m)-f_1(n)]+v_i}h_{i,N}\cdot T^{[f_i(n)-f_1(n)]+\tilde v_i}\overline{h}_{i,N}  \d\mu.
\end{split}
\end{equation*}
Hence \eqref{eqn_characteristic_factor_main_1-3} is the same as
\begin{equation}
\label{eqn_characteristic_factor_main_1-4}
\begin{split}
\lim_{N\to\infty}&  \Bigg|\frac{1}{W(N)} \sum_{n=1}^N  w(n) a_N(n+m)\overline{a}_N(n) 1_{A_v}(n)
\\
&\int \prod_{i=1}^k T^{[f_i(n+m)-f_1(n)]+v_i}h_{i,N}\cdot T^{[f_i(n)-f_1(n)]+\tilde v_i}\overline{h}_{i,N}  \d\mu\Bigg|=\, 0.
\end{split}
\end{equation}
Now let
$$
\tilde f_{i,m}(n)\coloneqq f_i(n)-f_{1}(n) \quad\text{and}\quad\tilde h_{i,N}\coloneqq T^{\tilde v_i}h_{i,N}\qquad\text{for each }i\in\{1,\dots,k\},
$$
and
$$
\tilde f_{k+i,m}(n)\coloneqq f_i(n+m)-f_{1}(n) \quad\text{and}\quad\tilde h_{k+i,N}\coloneqq T^{ v_i}\overline{h}_{i,N}\qquad\text{for each }i\in\{1,\dots,k\}.
$$
Also, set $\tilde{a}_N(n)\coloneqq a_N(n+m)\overline{a}_N(n) 1_{A_v}(n)$.
With this notation, \eqref{eqn_characteristic_factor_main_1-4} can be written as
\begin{equation}
\label{eqn_characteristic_factor_main_1-5}
\lim_{N\to\infty}\Bigg|\frac{1}{W(N)} \sum_{n=1}^N  w(n) \tilde{a}_N(n)\int \prod_{i=1}^{2k} T^{[\tilde{f}_{i,m}(n)]}\tilde{h}_{i,N} \d\mu\Bigg|=\, 0.
\end{equation}
Observe that $\lhk\tilde h_{2k,N}\rhk_s=\lhk h_k\rhk_s=0$.
Also, since $|f_i(t)|\ll |f_k(t)|$ for all $i=1,\ldots,k-1$, it follows that $|\tilde f_{i,m}(t)|\ll|\tilde f_{2k,m}(t)|$ for all $i=1,\ldots,2k-1$.
Moreover, since the degree of $f_k$ is at least $2$, the degree of $\tilde f_{2k,m}$ is at least $1$.
Thus \eqref{eqn_characteristic_factor_main_1-5} will follow by induction after we show that the characteristic vector $(\tilde m_1,\dots,\tilde m_\ell)$ of $\tilde \F=\{\tilde{f}_{1,m},\ldots,\tilde{f}_{2k,m}\}$ is strictly smaller than the characteristic vector $(m_1,\dots,m_\ell)$ of $\F=\{f_1,\ldots,f_k\}$.

Indeed, for each $f_i$ which is not equivalent to $f_{1}$, the functions $\tilde f_{i,m}$ and $\tilde f_{k+i,m}$ are equivalent to each other.
Moreover, if $f_i(t)$ and $f_j(t)$ are not equivalent to $f_{1}(t)$, then $\tilde f_{i,m}$ is equivalent to $\tilde f_{j,m}$ if and only if $f_i(t)$ is equivalent to $f_j(t)$.
Letting $d$ be the degree of $f_{1}(t)$, this shows that $\tilde m_j=m_j$ for all $j>d$.
Finally, if $f_i(t)$ is equivalent to $f_{1}(t)$, then both $f_i(t)-f_{1}(t)$ and $f_{i}(t+m)-f_{1}(t)$ have degree smaller than that of $f_i(t)$. This shows that $\tilde m_d<m_d$.
Therefore $(\tilde m_1,\dots,\tilde m_\ell)$ is strictly smaller than $(m_1,\dots,m_\ell)$ and we are done.
\end{proof}

\section{Proof of Theorem~\ref{thm_main}}
\label{sec_nilfac}
In this section we give the proof of Theorem~\ref{thm_main}. 
The first step is to use \cref{prop_characteristic_factor_main}, which was proved in the previous section, 
together with the structure theory of Host-Kra to reduce Theorem~\ref{thm_main} to the following special case:

\begin{Theorem}
\label{thm_F}
Theorem~\ref{thm_main} holds when $(X,\B,\mu,T)$ is an ergodic nilsystem.
\end{Theorem}

A measure preserving system is called an \define{$s$-step pro-nilsystem} if it is (isomorphic in the category of measure preserving systems to) an inverse limit of $s$-step nilsystems.
In other words, $(X,\mathcal{B},\mu,T)$ is an $s$-step pro-nilsystem if there exist $T$-invariant $\sigma$-algebras ${\mathcal B}_1\subset{\mathcal B_2}\subset\cdots\subset{\mathcal B}$ such that ${\mathcal B}=\bigcup {\mathcal B}_n$ and for all $n$, the system $(X,{\mathcal B}_n,\mu,T)$ is isomorphic to an $s$-step nilsystem.
The following result from \cite{HK05a} gives the relation between the uniformity seminorms and pro-nilsystems.
\begin{Theorem}[{\cite{HK05a}}]\label{thm_hk}
  Let $(X,\mathcal{B},\mu,T)$ be an ergodic system. For each $s\in\N$ there exists a factor ${\mathcal Z}_s\subset{\mathcal B}$ with the following properties:
\begin{enumerate}
  \item The measure preserving system $(X,{\mathcal Z}_s,\mu,T)$ is an $s$-step pro-nilsystem;
  \item a function $h\in L^\infty$ is measurable with respect to ${\mathcal Z}_s$ if and only if it belongs to the set
$\displaystyle \left\{h\in L^\infty(X):\int_Xh\cdot h'\d\mu=0\quad\forall h'\in L^\infty(X,{\mathcal B},\mu)\text{ with }\lhk h'\rhk_{s+1}=0 \right\}$.
\end{enumerate}
\end{Theorem}

\begin{proof}[Proof of Theorem~\ref{thm_main}]

Let $\Hardy$ be a Hardy field, let $f_1,\ldots,f_k\in\Hardy$ be of polynomial growth.
Let $W\in\Hardy$ be compatible with $f_1,\dots,f_k$, in the sense that Property \ref{property_P_W} holds.
Let $(X,\B,\mu,T)$ be an invertible measure preserving and let $h_1,\dots,h_k\in L^\infty(X)$ and $A\in\B$ with $\mu(A)>0$.
Using a usual ergodic decomposition argument, we may assume without loss of generality that the system is ergodic.

If $\lim_{t\to\infty}f_i(t)$ is finite for some $f_i$, then neither condition of Theorem~\ref{thm_main_combinatorial} holds so parts \ref{itm_thm_main_ii} and \ref{itm_thm_main_iii} hold vacuously. As every $f_i\in\Hardy$ is eventually monotone, this also implies that $[f_i(n)]$ is eventually constant, and hence part \ref{itm_thm_main_i} holds if it holds when $f_i$ is removed from $\{f_1,\dots,f_k\}$. Iterating this observation we may assume that $\lim_{t\to\infty}|f_i(t)|=\infty$ for every $i$.
A similar argument shows that we can assume that $\lim_{t\to\infty}|f_i(t)-f_j(t)|=\infty$ for every $i\neq j$.

Let $s\in\N$ be given by \cref{prop_characteristic_factor_main} and let ${\mathcal Z}_s$ be given by \cref{thm_hk}.
In view of Theorems \ref{prop_characteristic_factor_main} and \ref{thm_hk} we can assume that all the functions $h_i$ are measurable in ${\mathcal Z}_s$, which is equivalent to assuming that $(X,\B,\mu,T)$ is itself an $s$-step pro-nilsystem.
Therefore there exists a sequence of $\sigma$-algebras $\B_1\subset\B_2\subset\cdots\subset\B$ such that for each $\ell\in\N$ the system $(X,\B_\ell,\mu,T)$ is an ergodic nilsystem.

It now follows from \cref{thm_F} that the conclusions \ref{itm_thm_main_i} and \ref{itm_thm_main_iii} of Theorem~\ref{thm_main} hold after replacing each $h_i$ with $\E[h_i\mid\B_\ell]$, and that
$$\lim_{N\to\infty}\frac{1}{W(N)}\sum_{n=1}^N w(n) \int_XH_\ell\cdot T^{[ f_1(n) ]}H_\ell\cdots T^{[f_k(n)]}H_\ell\,>\,0$$
where $H_\ell=\E[1_A\mid\B_\ell]$ and $\E[\cdot\mid\B_\ell]$ denotes the conditional expectation on the $\sigma$-subalgebra $\B_\ell$.
Since $\E[h_i\mid\B_\ell]\to h_i$ and $H_\ell\to1_A$ in $L^p$ for every $p<\infty$ as $\ell\to\infty$, we conclude that the conclusions of Theorem~\ref{thm_main} also hold for $h_1,\dots,h_k$ and $A$.
\end{proof}

The rest of this section is devoted to proving \cref{thm_F}.

\paragraph{\bf Outline of the proof of \cref{thm_F}:}

The first step in the proof of \cref{thm_F} is a reduction to the case when the nilmanifold $X$ is of the form $X=G/\Gamma$ for a connected and simply connected nilpotent Lie group $G$.
These conditions on $G$ ensure that the nilsystem is embedable into a flow (i.e. a $\R$-action) on the same nilmanifold.
This reduction is achieved at the price of losing ergodicity of the $\Z$-action; nevertheless, the $\R$-flow is ergodic.
Then another reduction is performed, to avoid local obstructions.

These reductions allow us to reformulate \cref{thm_F} as a statement about uniform distribution of certain sequences in a nilmanifold, which
we prove by drawing from recent results obtained in \cite{Richter20arXiv}.

\subsection{Reducing to connected and simply connected groups}
The first step in the proof of \cref{thm_F} is essentially a reduction to the case when $X$ is a nilmanifold $G/\Gamma$ where the corresponding nilpotent Lie group $G$ is connected and simply connected.
This is a somewhat technical and by now standard reduction but which is crucial in order to deal with the rounding function and to use the equidistribution results from \cite{Richter20arXiv}.
It turns out that the additional structure of nilsystems allows us to upgrade $L^2$ convergence into pointwise convergence.
We now state this case of \cref{thm_F} as a separate theorem.

\begin{Theorem}
\label{thm_connectedsimplyconnected}
Let $f_1,\ldots,f_k$ be functions of polynomial growth from a Hardy field $\Hardy$ and let $W\in\Hardy$ be any function
with $1\prec W(t)\ll t$ and such that $f_1,\ldots,f_k$ satisfy Property \ref{property_P_W}.
Then for any nilsystem $(X,\B_X,\mu_X,T)$, where $X=G/\Gamma$ and $G$ is connected and simply connected, we have:
\begin{enumerate}
[label=(\roman{enumi}),ref=(\roman{enumi}),leftmargin=*]
\item
\label{itm_thm_connsimplyconn_i}
For any $h_1,\ldots,h_k\in C(X)$ and $x\in X$ the limit
\begin{equation*}
\lim_{N\to\infty}\frac{1}{W(N)}\sum_{n=1}^N w(n) \prod_{i=1}^k T^{[f_i(n)]}h_i(x)
\end{equation*}
exists.
\item
\label{itm_thm_connsimplyconn_iii}
If $f_1,\ldots,f_k$ satisfy condition \ref{itm_cond_2} of Theorem~\ref{thm_main_combinatorial} then for any $h_1,\ldots,h_k\in C(X)$ and any $x\in X$
\begin{equation*}
\lim_{N\to\infty}\frac{1}{W(N)}\sum_{n=1}^N w(n) \prod_{i=1}^k T^{[f_i(n)]}h_i(x)\,=\, \prod_{i=1}^k h_i^*(x)
\end{equation*}
where $h_i^*(x):=\lim_{N\to\infty}\tfrac1N\sum_{n=1}^Nh_i(T^nx)$.\footnote{Recall (e.g. from \cite{AGH63}) that for a nilsystem $(X,T)$ and a continuous function $h\in C(X)$ the limit $\lim_{N\to\infty}\tfrac1N\sum_{n=1}^Nh(T^nx)$ exists at every point $x$.}
\item
\label{itm_thm_connsimplyconn_ii}
If $f_1,\ldots,f_k$ satisfy condition \ref{itm_cond_1} of Theorem~\ref{thm_main_combinatorial} then for any $A\in\B$ with $\mu(A)>0$ we have
\begin{equation*}
\lim_{N\to\infty}\frac{1}{W(N)}\sum_{n=1}^N w(n) \mu\Big(A\cap T^{-[ f_1(n) ]}A\cap\ldots\cap T^{-[f_k(n)]}A\Big)\,>\,0.
\end{equation*}
\end{enumerate}
\end{Theorem}

\begin{proof}[Proof of \cref{thm_F} assuming \cref{thm_connectedsimplyconnected}]
Let $(X,\B_X,\mu_X,T)$ be an ergodic nilsystem.
In view of \cref{prop_malcev} we can represent $X=G/\Gamma$ as a closed subsystem of another nilsystem $(Z,\B_Z,\mu_Z,S)$ where $Z=\hat G/\hat\Gamma$ for a connected and simply connected nilpotent Lie group $\hat G$ which contains $G$ as a closed and rational subgroup.
Letting $\tilde G$ be the connected component of the identity in $G$ and letting $a\in G$ be such that $T$ corresponds to left multiplication by $a$, we also have that $G=\overline{a^\Z\tilde G}$ and $\hat G=\overline{a^\R\tilde G}$.
If $G=\hat G$ then $G$ is connected and simply connected and the desired conclusion follows directly from \cref{thm_connectedsimplyconnected}.
From now on suppose that $G\neq\hat G$.

To establish parts \ref{itm_thm_main_i} and \ref{itm_thm_main_iii} of Theorem~\ref{thm_main} for the nilsystem $(X,\B_X,\mu_X,T)$, it suffices to consider continuous functions $h_i$, for they are dense in $L^2$.
By the Tietze-Urysohn extension theorem, any continuous function $h\in C(X)$ can be extended to a continuous function on $Z$.
The conclusions now follow directly from parts \ref{itm_thm_connsimplyconn_i} and \ref{itm_thm_connsimplyconn_iii} of \cref{thm_connectedsimplyconnected}.

Before proving part \ref{itm_thm_main_ii} of Theorem~\ref{thm_main} for the nilsystem $(X,\B_X,\mu_X,T)$, we need some facts about the relation between $X$ and $Z$.
For every $y\in X$ the set $D:=\{t\in\R:a^ty\in X\}=\{t\in\R:a^t\in G\}$ forms a closed subgroup of $\R$.
We can't have $D=\R$ for this would imply $G=\hat G$, so $D$ must be discrete.
Notice that $D$ is also independent of $y$.
Let $c=\min\{t>0:t\in D\}$. 
Let $R:=S^c:X\to X$.
Since $1\in D$, it follows that $c=\tfrac1m$ for some $m\in\N$ and hence $R^m=T$ .

We claim that $\mu_Z=\tfrac1c\int_0^cS^t_*\mu_X\d t$.
Indeed, since $T$ acts ergodically on $X$, then so does $R$.
Therefore the system $(X,R)$ is uniquely ergodic and hence $\mu_X=\lim_{N\to\infty}\frac1N\sum_{n=1}^N\delta_{R^nx}$.
Since $\hat G=\overline{a^\R\tilde G}=\overline{\bigcup_{t\in[0,1)}a^tG}$, we have $Z=\overline{\bigcup_{t\in[0,1)}S^tX}$ and thus the flow $(S^t)_{t\in\R}$ on $Z$ is transitive and hence uniquely ergodic.
It follows that
$$\mu_Z
=
\lim_{N\to\infty}\frac1N\int_0^N\delta_{S^tx}\d t
=
\lim_{N\to\infty}\frac1N\sum_{n=1}^{mN}\int_0^c\delta_{S^tR^nx}\d t
=\tfrac1c\int_0^cS^t_*\mu_X\d t.$$

Now, fix $A\subset X$ with $\mu_X(A)>0$.
Let $B=\bigcup_{t\in[0,c)}S^{t}A$.
It follows that $S^{-t}B\cap X=A$ for every $t\in[0,c)$ and hence $\mu_Z(B)=m\int_0^c\mu_X(a^{-t}B)\d t=\mu_X(A)>0$.
In view of part \ref{itm_thm_connsimplyconn_ii} of \cref{thm_connectedsimplyconnected} (and the choice of $W$) applied to $(Z,\B_Z,\mu_Z,S)$ we deduce that
$$\lim_{N\to\infty}\frac1{W(N)}\sum_{n=1}^Nw(n)\mu_Z\big(B\cap S^{-[f_1(n)]}B\cap\cdots\cap S^{-[f_k(n)]}B\big)>0$$
as long as the functions $f_i$ satisfy either condition \ref{itm_cond_1} or \ref{itm_cond_2} of Theorem~\ref{thm_main_combinatorial}.
Since $X$ is $T$ invariant, for any $n_1,\dots,n_k\in\Z$ and any $t\in[0,c)$ we have
$$S^{-t}\big(B\cap S^{n_1}B\cap\cdots\cap S^{n_k}B\big)\cap X=A\cap T^{n_1}A\cap\cdots\cap T^{n_k}A.$$
Therefore, for every $n\in\N$
\begin{equation*}
\begin{split}
\mu_Z\big(B\cap S^{-[f_1(n)]}B\cap & \dots\cap S^{-[f_k(n)]}B\big)
\\&=
\frac1c\int_0^c\mu_X\big(S^{-t}\big(B\cap S^{-[f_1(n)]}B\cap\cdots\cap S^{-[f_k(n)]}B\big)\big)\d t
\\&=
\frac1c\int_0^c\mu_X\big(A\cap S^{-[f_1(n)]}A\cap\cdots\cap S^{-[f_k(n)]}A\big)\d t
\\&=
\mu_X\big(A\cap S^{-[f_1(n)]}A\cap\cdots\cap S^{-[f_k(n)]}A\big)
\end{split}
\end{equation*}
and the desired conclusion follows.
\end{proof}

\subsection{Dealing with rational polynomials}

Next we perform another simplification, this time restricting the class of allowed functions $f_i$.
\begin{Theorem}
\label{thm_norationalpolynomials}
\cref{thm_connectedsimplyconnected} holds if we further assume that each function $f_i$ satisfies the condition
\begin{equation}\label{eq_rationalcondition}
  \text{If }|f_i(t)-p(t)|\to0\text{ for some }p\in\R[t]\text{ with }p-p(0)\in\Q[t]\text{, then }p-p(0)\in\Z[t].
\end{equation}
\end{Theorem}

To reduce \cref{thm_connectedsimplyconnected} to \cref{thm_norationalpolynomials} we need the following lemma.

\begin{Lemma}\label{lemma_WaveragesoverAPs}
  Let $a(n)$ be a bounded sequence (of real or complex numbers, or vectors of a Banach space), let $R\in\N$ and let $\Hardy$ be a Hardy field, $W\in\Hardy$ satisfying $1\prec W(t)\ll t$ and $w:=\Delta W$.
  Then
  $$\lim_{N\to\infty}\left[\frac{1}{W(RN)}\sum_{n=1}^{RN} w(n)a(n)
  -
  \frac1R\sum_{d=0}^{R-1}\frac{1}{W(N)}\sum_{n=1}^N w(n)a(Rn+d)\right]=0.$$
\end{Lemma}
\begin{proof}
  Because $W\in\Hardy$, $\lim_{t\to\infty} W(t)/t$ exists and, because $W(t)\ll t$, that limit is in $[0,\infty)$.
  If it is nonzero, then $L:=\lim \frac{W(Rt+d)}{W(t)}=\lim \frac{W(Rt+d)}{Rt+d}\frac t{W(t)}\frac{Rt+d}t=R$, and if $\lim W(t)/t=0$ then $W(t)/t$ is eventually decreasing, so $L\leq R$.
  Since $W$ is eventually increasing (because $1\prec W(t)$) we have $L\geq1$.
  Next it follows from L'H\^opital's rule and the mean value theorem that $w(Rt+d)/w(t)\to\frac LR$.

  We now have
  \begin{eqnarray*}\frac{1}{W(RN)}\sum_{n=1}^{RN} w(n)a(n)
  &=&
  \frac{1+o(1)}{LW(N)}\sum_{d=0}^{R-1}\sum_{n=1}^{N} w(Rn+d)a(Rn+d)
  \\&=&
\sum_{d=0}^{R-1}\frac{1+o(1)}{LW(N)}\sum_{n=1}^{N} \frac LRw(n)a(Rn+d)
  \\&=&
\frac{1+o(1)}R\sum_{d=0}^{R-1}\frac1{W(N)}\sum_{n=1}^{N}w(n)a(Rn+d).
  \end{eqnarray*}
\end{proof}
\begin{proof}[Proof of \cref{thm_connectedsimplyconnected} assuming \cref{thm_norationalpolynomials}]
After reordering the $f_i$'s if necessary, we can assume that the functions $f_1,\dots,f_m$ do not satisfy condition \eqref{eq_rationalcondition}, whereas the functions $f_{m+1},\dots,f_k$ do.
For each $i=1,\dots,m$ let $p_i\in\R[t]$ be such that $p_i-p_i(0)\in\Q[t]$ and $|f_i(t)-p_i(t)|\to0$.
Let $R$ be a common denominator for all the coefficients of all the polynomials $p_i-p_i(0)$, for $i=1,\dots,m$.

Observe that the polynomials $p_{i,d}\colon t\mapsto p_i(Rt+d)-p_i(d)$ have integer coefficients for any fixed $d\in\N$ and any $i\leq m$.
Therefore, for each $d\in\N$ and $i\in\{1,\dots,k\}$ the function $f_{i,d}\colon t\mapsto f_i(Rt+d)$ satisfies condition \eqref{eq_rationalcondition}.
Parts \ref{itm_thm_connsimplyconn_i} and \ref{itm_thm_connsimplyconn_iii} of \cref{thm_connectedsimplyconnected} now follow from \cref{thm_norationalpolynomials} and \cref{lemma_WaveragesoverAPs}.

Finally, we prove part \ref{itm_thm_connsimplyconn_ii}.
Suppose that the functions $f_1,\ldots, f_k$ satisfy condition \ref{itm_cond_1} of Theorem~\ref{thm_main_combinatorial} and let $q_1,\dots,q_\ell\in\Z[t]$ be jointly intersective and such that $\poly(f_1,\dots,f_k)\subset\sp(q_1,\dots,q_\ell)$.
By the pigeonhole principle, for some $d\in\{0,\dots,R-1\}$ the polynomials $\tilde q_i:t\mapsto q_i(Rt+d)$, $i=1,\dots,\ell$ are jointly intersective.
Therefore the functions $f_{1,d},\ldots,f_{k,d}$ also satisfy condition \ref{itm_cond_1} of Theorem~\ref{thm_main_combinatorial}.
Then using \cref{lemma_WaveragesoverAPs} and \cref{thm_norationalpolynomials} we have
\begin{equation*}
\begin{split}
\lim_{N\to\infty}\frac{1}{W(N)}\sum_{n=1}^N w(n) & \mu\left(\bigcap_{i=0}^kT^{-[f_i(n)]}A\right)
\\
&\geq~
\frac1R\lim_{N\to\infty}\frac{1}{W(N)}\sum_{n=1}^Nw(n)\mu\left(\bigcap_{i=0}^kT^{-[f_{i,d}(n)]}A\right)
>
0,
\end{split}
\end{equation*}
where, to simplify the expression, we used $f_0=f_{0,d}\equiv0$.
\end{proof}

\subsection{Reduction to a statement about uniform distribution}
\label{sec_intermediate_result}

Thus far, we have reduced \cref{thm_F} to \cref{thm_norationalpolynomials}. Our next goal is to reduce \cref{thm_norationalpolynomials} to the following result:

\begin{Theorem}\label{thm_lastuniformdistribution}
Let $G$ be a connected and simply connected nilpotent Lie group, let $\Gamma\subset G$ be a uniform and discrete subgroup, let $X=G/\Gamma$, and let $a\in G$.
  Let $f_1,\dots,f_k\in\Hardy$ satisfying property \ref{property_P_W} for some $W\in\Hardy$ with $1\prec W(t)\ll t$ and property \eqref{eq_rationalcondition} from \cref{thm_norationalpolynomials}.
Then there exists a Borel probability measure $\nu$ on $X^k$ such that the sequence
  \begin{equation}\label{eq_thm_removingfloors}
    n\mapsto\Big(a^{[f_1(n)]},a^{[f_2(n)]},\dots,a^{[f_k(n)]}\Big)\Gamma^k
  \end{equation}
  is uniformly distributed  with respect to $\nu$ and $W$-averages.

Moreover, if condition \ref{itm_cond_1} of Theorem~\ref{thm_main_combinatorial} is satisfied then the point $1_{X^k}=1_{G^k}\Gamma^k$ belongs to the support of $\nu$, and if condition \ref{itm_cond_2} of Theorem~\ref{thm_main_combinatorial} is satisfied then
$\nu$ is the Haar measure on the subnilmanifold $Y^k\subset X^k$, where $Y=\overline{\{a^n\Gamma:n\in\Z\}}$.

\end{Theorem}

\begin{proof}[Proof of \cref{thm_norationalpolynomials} assuming \cref{thm_lastuniformdistribution}]
Let $G$ be a connected and simply connected nilpotent Lie group and let $\Gamma\subset G$ be a uniform and discrete subgroup such that $X=G/\Gamma$ and $T:X\to X$ is given by left multiplication by some element $a\in G$.
Our goal is to show that: 

\begin{enumerate}
[label=(\roman{enumi}),ref=(\roman{enumi}),leftmargin=*]
\item
For any $h_1,\ldots,h_k\in C(X)$ and $x\in X$ the limit
\begin{equation*}
\lim_{N\to\infty}\frac{1}{W(N)}\sum_{n=1}^N w(n) \prod_{i=1}^k T^{[f_i(n)]}h_i(x)
\end{equation*}
exists.
\item
If $f_1,\ldots,f_k$ satisfy condition \ref{itm_cond_2} of Theorem~\ref{thm_main_combinatorial} then for any $h_1,\ldots,h_k\in C(X)$ and any $x\in X$
\begin{equation*}
\lim_{N\to\infty}\frac{1}{W(N)}\sum_{n=1}^N w(n) \prod_{i=1}^k T^{[f_i(n)]}h_i(x)\,=\, \prod_{i=1}^k h_i^*(x)
\end{equation*}
where $h_i^*(x):=\lim_{N\to\infty}\tfrac1N\sum_{n=1}^Nh_i(T^nx)$.
\item
If $f_1,\ldots,f_k$ satisfy condition \ref{itm_cond_1} of Theorem~\ref{thm_main_combinatorial} then for any $A\in\B$ with $\mu(A)>0$ we have
\begin{equation*}
\lim_{N\to\infty}\frac{1}{W(N)}\sum_{n=1}^N w(n) \mu\Big(A\cap T^{-[ f_1(n) ]}A\cap\ldots\cap T^{-[f_k(n)]}A\Big)\,>\,0.
\end{equation*}
\end{enumerate}

For each $x\in X$, let $g_x\in G$ be such that $x=g_x\Gamma$.
For any $h_1,\dots,h_k\in C(X)$ we have
$$\prod_{i=1}^kh_i\big(a^{[f_i(n)]}x\big)
=
\prod_{i=1}^kh_i\Big(g_x\big(g_x^{-1}ag_x\big)^{[f_i(n)]}\Gamma\Big).$$
Now let $\nu_x$ be the measure on $X^k$ given by \cref{thm_lastuniformdistribution} with $g_x^{-1}ag_x$ in the role of $a$ and define $H_x(x_1,x_2,\dots,x_k):=h_1(g_xx_1)h_2(g_xx_2)\cdots h_k(g_xx_k)$.
We then have
\begin{equation}\label{eq_proof_thm_multirecnilsystem1}
\lim_{N\to\infty}\frac{1}{W(N)}\sum_{n=1}^N w(n)\prod_{i=1}^kh_i\big(a^{[f_i(n)]}x\big)
=
\int_{X^k} H_x\d\nu_x.
\end{equation}
In particular, this proves part \ref{itm_thm_connsimplyconn_i}. 
To prove part \ref{itm_thm_connsimplyconn_iii}, suppose that condition \ref{itm_cond_2} holds.
Then $\nu_x$ is the Haar measure on $Y_x^k$, where $Y_x=\overline{\{g_x^{-1}a^ng_x\Gamma:n\in\Z\}}$.
The Haar measure $\mu_{Y_x}$ on $Y_x$ can be described by $\int_{Y_x}h\d\mu_{Y_x}=\lim_{N\to\infty}\frac1N\sum_{n=1}^Nf(g_x^{-1}a^ng_x\Gamma)$, and hence for each $i=1,\dots,k$,
$$\int_{Y_x}h_i(g_xy)\d\mu_{Y_x}(y)=\lim_{N\to\infty}\frac1N\sum_{n=1}^Nh_i(a^nx)=h_i^*(x).$$
Therefore
$$\int_{X^k} H_x\d\nu_x=\prod_{i=1}^k\int_{Y_x}h_i(g_xx_i)\d\mu_{Y_x}(x_i)=\prod_{i=1}^kh_i^*(x).$$
This together with \eqref{eq_proof_thm_multirecnilsystem1} proves part \ref{itm_thm_connsimplyconn_iii}.

To prove part \ref{itm_thm_connsimplyconn_ii} it is enough to show that for any continuous non-negative function $h\in\Cont(X)$ with $\int_Xh\d\mu>0$ we have
$$\lim_{N\to\infty}\frac{1}{W(N)}\sum_{n=1}^N w(n)\int_Xh(x)\cdot\prod_{i=1}^kh\big(a^{[f_i(n)]}x\big)\d\mu_X(x)>0.$$
In view of \eqref{eq_proof_thm_multirecnilsystem1} and the dominated convergence theorem, it suffices to show that
\begin{equation}\label{eq_multiplerecurrence}
  \int_Xh(x)\int_{X^k}H_x\d\nu_x\d\mu_X(x)>0,
\end{equation}
where $H_x$ is the same function as above, only with $h_1=\ldots=h_k=h$.
Observe that if $h(x)>0$ for some $x\in X$, then $H_x(1_{G^k}\Gamma^k)=h(x)^k>0$.
Since $H_x$ is a continuous function and $1_{G^k}\Gamma^k$ is in the support of $\nu_x$ for every $x\in X$, it follows that whenever $h(x)>0$, also $\int_{X^k}H_x\d\nu_x>0$.
Since $\int_Xh\d\mu_X>0$ we have that $\mu_X(\{x\in X:h(x)>0\})>0$, so \eqref{eq_multiplerecurrence} holds.
\end{proof}

\subsection{Removing the rounding function}
\label{sec_another_reduciton}

Next, let us reduce \cref{thm_lastuniformdistribution} to a result that no longer involves the bracket function $[.]\colon\R\to\Z$.

\begin{Theorem}\label{thm_nilequidistribution_2_1_new}
Let $G$ be a connected and simply connected nilpotent Lie group, let $\Gamma\subset G$ be a uniform and discrete subgroup, let $b\in G$ and $X=G/\Gamma$.
Let $\Hardy$ be a Hardy field and let $W\in\Hardy$ satisfy $1\prec W(t)\ll t$.
For any finite collection of functions $f_1,\ldots,f_k\in\Hardy$ of polynomial growth satisfying property \ref{property_P_W} and property \eqref{eq_rationalcondition} there exists a Borel probability measure $\nu$ on $X^k$ such that the sequence $x_n=\big(b^{f_1(n)},\dots,b^{f_k(n)}\big)\Gamma^k$ is uniformly distributed with respect to $\nu$ and $W$-averages.

Moreover, if condition \ref{itm_cond_1} of Theorem~\ref{thm_main_combinatorial} is satisfied then the point $1_{X^k}=1_{G^k}\Gamma^k$ belongs to the support of $\nu$, and if condition \ref{itm_cond_2} of Theorem~\ref{thm_main_combinatorial} is satisfied then
$\nu$ is the Haar measure on the subnilmanifold $Y^k$ where $Y\subset X$ is defined by $Y=\overline{\{b^t\Gamma:t\in\R\}}$.
\end{Theorem}

\begin{proof}[Proof of \cref{thm_lastuniformdistribution} assuming \cref{thm_nilequidistribution_2_1_new}]
Let $\tilde G=G\times\R$, let $\tilde\Gamma=\Gamma\times\Z$, let $\tilde X=\tilde G/\tilde\Gamma$ and let $b=(a,1)\in\tilde G$.
Let $\rho(t)=t-[t]$ for $t\in\R$ and let $\pi\colon \tilde X^k\to X^k$ be the map
$$\pi\big((g_1,t_1,\dots,g_k,t_k)\tilde\Gamma^k\big)= (a^{-\rho(t_1)}g_1,\dots,a^{-\rho(t_k)}g_k)\Gamma^k.$$
Observe that $\pi$ is well defined (i.e., the choice of the co-set representative does not matter in the definition of $\pi$) and that $\pi\big((b^{t_1},\dots,b^{t_k})\tilde\Gamma^k\big)=(a^{[t_1]},\dots,a^{[t_k]})\Gamma^k$ for every $t_1,\dots,t_k\in\R$.
However, we alert the reader that $\pi$ is not a continuous map.

From \cref{thm_nilequidistribution_2_1_new} it follows that the sequence $x_n:=(b^{f_1(n)},b^{f_2(n)},\dots,b^{f_k(n)})\tilde\Gamma^k$ is uniformly distributed with respect to $\tilde{\nu}$ and $W$-averages, where $\tilde{\nu}$ is some Borel probability measure on $\tilde{X}^k$.
We now consider two cases separately.
\begin{itemize}
  \item Case I: the map $\pi$ is $\tilde{\nu}$-a.e.\ continuous.

In this case, for any $H\in C(X^k)$ also $H\circ \pi$ is $\tilde{\nu}$-a.e.\ continuous and hence, since
$$H\bigg(\Big(a^{[f_1(n)]},a^{[f_2(n)]},\dots,a^{[f_k(n)]}\Big)\Gamma^k\bigg)= H\circ\pi (x_n),$$
it follows that the sequence defined in \eqref{eq_thm_removingfloors} is uniformly distributed with respect to the pushforward $\nu:=\pi_*\tilde{\nu}$ and $W$-averages.

If condition \ref{itm_cond_1} of Theorem~\ref{thm_main_combinatorial} is satisfied then \cref{thm_nilequidistribution_2_1_new} also implies that the point $1_{\tilde X^k}$ is in the support of $\tilde\nu$.
Since $\rho$ is continuous at $0$, the map $\pi$ is continuous at $1_{\tilde X^k}$ and, because $\nu=\pi_*\tilde\nu$, it follows that
the point $1_{X^k}=\pi(1_{\tilde X^k})$ belongs to the support of $\nu$.

If condition \ref{itm_cond_2} of Theorem~\ref{thm_main_combinatorial} is satisfied then \cref{thm_nilequidistribution_2_1_new} also implies that $\tilde\nu$ is the Haar measure on the subnilmanifold $\tilde Y^k$ where $\tilde Y\subset\tilde X$ is defined by $\tilde Y=\overline{\{b^t\Gamma:t\in\R\}}$.
Since $Y:=\{a^n\Gamma:n\in\Z\}$ satisfies $Y^k=\pi(\tilde Y^k)$ it follows that $\nu=\pi_*\tilde\nu$ is the Haar measure on $Y^k$.

\item Case II: $\pi$ is not $\tilde{\nu}$-a.e. continuous.

Then there exists $i\in\{1,\dots,k\}$ such that the set $D_i:=\{(g_1,t_1,\dots,g_k,t_k)\tilde\Gamma^k:\{t_i\}=1/2\}$ has $\tilde{\nu}(D_i)>0$, where $\{x\}=x-\lfloor x\rfloor$ denotes the fractional part.
After reordering we can assume $\mu(D_i)>0$ for $i=1,\dots,k_0$ and $\mu(D_i)=0$ for $i>k_0$.
Since the sequence $(x_n)_{n\in\N}$ is uniformly distributed with respect to $W$-averages and with respect to $\tilde{\nu}$, it follows that for each $i\leq k_0$,
$$
\lim_{\epsilon\to0}d_W\Big(\big\{n\in\N:\big\|f_i(n)-\tfrac{1}{2}\big\|_\T<\epsilon\big\}\Big)>0,
$$
where $\|x\|_\T= |x- [x]|$ is the distance to the closest integer.
In view of \cref{Cor_2claims}, there exists a polynomial $p_i\in\Q[x]$ such that $|f_i(n)-p_i(n)|\to0$ as $n\to\infty$.
In view of property \eqref{eq_rationalcondition}, for each $i\leq k_0$ we in fact have $p_i(x)-p_i(0)\in\Z[t]$.
This implies in particular that Case II is incompatible with condition \ref{itm_cond_2} of Theorem~\ref{thm_main_combinatorial}.

It now follows that the sequence $\big\{p_i(n)\big\}$ is constant and hence equals $1/2$, and therefore the sequence $\big\{f_i(n)\big\}$ converges to $1/2$.
This implies that all the accumulation points of the sequence $(x_n)_{n\in\N}$ lie inside $D_i$.
Therefore $\tilde{\nu}(D_i)=1$ for all $i\leq k_0$.
Hence $\supp\tilde{\nu}$ is a subset of $D:=\bigcap_{i=1}^{k_0}D_i$.
Since $\tilde{\nu}(D_i)=0$ for all $i> k_0$, it follows that the restriction of $\pi$ to $D$ is $\tilde{\nu}$-a.e. continuous\footnote{This may sound somewhat paradoxical since $\pi$ is discontinuous at every point in $D$, but this is the same phenomenon exhibited by the map $f:(x,y)\mapsto\lfloor x\rfloor$, which is discontinuous at every point in the vertical line $L=\{(0,y):y\in\R\}$, but whose restriction $\pi|_L$ is constant and hence continuous.}.
Therefore, we now restrict our attention to $D$.

Let $y_n=(a^{f_1(n)},1/2,\dots,a^{f_{k_0}(n)},1/2,b^{f_{k_0+1}(n)},\dots,b^{f_k(n)})\tilde\Gamma^k\in D\subset\tilde X^k$ and note that $d(x_n,y_n)\to0$ as $n\to\infty$ outside of a set of zero density, where $d$ is any compatible metric on $\tilde X^k$.
It follows that $(y_n)_{n\in\N}$ is also uniformly distributed with respect to $\tilde{\nu}$ and $W$-averages.
Unfortunately, it is not necessarily the case that $\pi(y_n)$ is getting close to $((a^{[f_1(n)]},\dots,a^{[f_k(n)]})\Gamma^k)$ (see \cref{example_annoying1} below).

Since for every $i\leq k_0$, the function $x\mapsto f_i(x)-p_i(x)$ is in $\Hardy$, it eventually stops changing sign.
Recalling that $p_i(n)-1/2$ is always an integer and $\rho$ has only jump discontinuities at $1/2+\Z$, we deduce that $\rho(f_i(n))=\rho(f_i(n)-p_i(n)+1/2)$ converges as $n\to\infty$ (to either $1/2$ or $-1/2$).
Let
\begin{equation}\label{eq_proof_thm_lastuniformdistribution}
c_i:=\rho\left(\frac12\right)
-\lim_{n\to\infty}\rho\big(f_i(n)\big)
\in\{-1,0\}.
\end{equation}
To ease the notation, set $c_i=0$ for all $i>k_0$.
Let $\psi:\tilde X^k\to X^k$ be the map obtained by composing $\pi$ with translation by the element $(a^{c_1},\dots,a^{c_k})\in G^k$, and let $\tilde\psi=\psi|_D$.
We claim that the sequence $(a^{[f_1(n)]},\dots,a^{[f_k(n)]})\Gamma^k$ is uniformly distributed with respect to the measure $\nu:=\tilde\psi_*\tilde{\nu}$ and $W$-averages.

Since $[f_i(n)]=f_i(n)-\rho(f_i(n))$, in view of \eqref{eq_proof_thm_lastuniformdistribution} after unraveling the definitions we have that
\begin{equation}\label{eq_distancezero}
\lim_{n\to\infty}d\Big(\big(a^{[f_1(n)]},\dots,a^{[f_k(n)]}\big)\Gamma^k,\tilde\psi(y_n)\Big)=0.
\end{equation}
If $H\in C(X^k)$, then $H\circ\tilde\psi$ is $\tilde{\nu}$-a.e. continuous, and therefore, from \eqref{eq_distancezero} and the fact that $(y_n)_{n\in\N}$ is  uniformly distributed with respect to $\tilde{\nu}$ and $W$-averages, we deduce that
\begin{eqnarray*}
  \int_{X^k}H\d\nu
  &=&
  \int_DH\circ\tilde\psi\d\mu_Y
  \\&=&
  \lim_{N\to\infty}\frac1{W(N)}\sum_{n=1}^Nw(n)H(\tilde\psi(y_n))
  \\&=&
  \lim_{N\to\infty}\frac1{W(N)}\sum_{n=1}^Nw(n) H\bigg(\Big(a^{[f_1(n)]},\dots,a^{[f_k(n)]}\Big)\Gamma^k\bigg),
\end{eqnarray*}
which proves the claim.

Finally, if $f_1,\dots,f_k$ satisfy condition \ref{itm_cond_1} of Theorem~\ref{thm_main_combinatorial}, then arguing as in Case I we see that the point $1_{G^k}\Gamma^k= 1_{X^k}$ belongs to the support of $\nu$.
Since we already saw above that, in Case II, condition \ref{itm_cond_2} of Theorem~\ref{thm_main_combinatorial} can not hold, this finishes the proof.
\end{itemize}

\begin{Example}\label{example_annoying1}
  Let $f_1(n)=n$, $f_2(n)=n-1/n$, $W(n)=n$, $[\cdot]=\lfloor\cdot\rfloor$, $G/\Gamma=\T=\R/\Z$ and $a\in\R\setminus\Q$ arbitrary.
  Then $\tilde G=\R^2$, $\tilde\Gamma=\Z^2$ and $\pi:\T^4\to\T^2$ is given by $\pi(x_1,t_1,x_2,t_2)=(x_1-a\{t_1\},x_2-a\{t_2\})$ and $(b^{f_1(n)},b^{f_2(n)})\Z^4$ equidistributes on $Y=\{(t,0,t,0):t\in\T\}$.
  In this case $\pi$ is discontinuous on $Y$, so it falls into Case II of the proof.
  Indeed $\pi_*\mu_Y$ is the Haar measure on the diagonal $\{(x,x):x\in\T\}\subset\T^2$, but the sequence $(\lfloor n\rfloor a,\lfloor n-1/n\rfloor a)$ is uniformly distributed on the set $\{(x,x-a):x\in\T\}\subset\T^2$.
\end{Example}
\end{proof}

\subsection{A Proof of \cref{thm_nilequidistribution_2_1_new}}

The purpose of this subsection is to give a proof of \cref{thm_nilequidistribution_2_1_new}. For this proof, we need two of the main results from \cite{Richter20arXiv}.

\begin{Theorem}[{\cite[Theorem E]{Richter20arXiv}}]
\label{thm_E}
Let $G$ be a connected and simply connected nilpotent Lie group, $\Gamma$ a uniform and discrete subgroup of $G$, and $\Hardy$ a Hardy field. Assume $W\in\Hardy$ satisfies $1\prec W(t)\ll t$ and $f_1,\ldots,f_k\in \Hardy$ satisfy property \ref{property_P_W}.
Let
$$
v(n)\coloneqq a_1^{f_1(n)}\cdot\ldots\cdot a_k^{f_k(n)},\qquad\forall n\in\N,
$$
where $a_1,\ldots,a_k\in G$ are commuting.
Then there exist $q\in\N$, a closed and connected subgroup $H$ of $G$, and points $x_0,x_1,\ldots,x_{q-1}\in X$ such that $Y_r\coloneqq H x_r$ is a closed sub-nilmanifold of $X$ and $(v(qn+r)\Gamma)_{n\in\N}$ is uniformly distributed with respect to $\mu_{Y_r}$ and $W$-averages for all $r=0,1,\ldots,q-1$.
\end{Theorem}

Given an element $a$ in a simply connected nilpotent Lie group $G$ we write $\dom(a)$ for the set of all $t\in\R$ for which $a^t$ is a well defined element of the group. For example, a rational number $\frac{r}{q}$ with $\gcd(r,q)=1$ belongs to $\dom(a)$ if and only if there exists $b\in G$ such that $b^q=a^r$. Since $G$ is assumed to be simply connected, if such a $b$ exists then it is unique. It also follows from the assumption of $G$ being simply connected that $\dom(a)=\R$ if and only if $a\in G^\circ$.

Given a connected nilmanifold $X=G/\Gamma$, the \define{maximal factor torus} of $X$ is the quotient $[G^\circ,G^\circ]\backslash X$, where $G^\circ$ is the identity component of $G$. We will use $\vartheta\colon X\to [G^\circ,G^\circ]\backslash X$ to denote the factor map from $X$ onto $[G^\circ,G^\circ]\backslash X$.
Note that the maximal factor torus is the torus of maximal dimension that is a factor of the nilmanifold $X$.

\begin{Theorem}[{\cite[Theorem D]{Richter20arXiv}}]\label{thm_D}
Let $G$ be a simply connected nilpotent Lie group, $\Gamma$ a uniform and discrete subgroup of $G$, and assume the nilmanifold $X=G/\Gamma$ is connected. Let $\Hardy$ be a Hardy field, $W$ a function in $\Hardy$ satisfying $1\prec W(t)\ll t$, and $f_1,\ldots,f_k\in\Hardy$ functions in $\Hardy$ satisfying property \ref{property_P_W}.
Consider
$$
v(n)\coloneqq a_1^{f_1(n)}\cdot\ldots\cdot a_k^{f_k(n)},\qquad\forall n\in\N,
$$
where $a_1,\ldots,a_k\in G$ are commuting, and $f_i(\N)\subset \dom(a_i)$ for all $i=1,\ldots,k$.
Then the following are equivalent:
\begin{enumerate}
[label=(\roman{enumi}),ref=(\roman{enumi}),leftmargin=*]
\item
the sequence $(v(n)\Gamma)_{n\in\N}$ is uniformly distributed with respect to $W$-averages in the nilmanifold $X=G/\Gamma$.
\item
The sequence $(\vartheta(v(n)\Gamma))_{n\in\N}$ is uniformly distributed with respect to $W$-averages in the maximal factor torus $[G^\circ,G^\circ]\backslash X$.
\end{enumerate}
\end{Theorem}

We actually need the following corollary of \cref{thm_D}.

\begin{Corollary}\label{thm_nilequidistribution_2_2}
Let $G$ be a simply connected nilpotent Lie group, let $\Gamma\subset G$ be a uniform and discrete subgroup, and suppose $X=G/\Gamma$ is connected. Let $a_1,\ldots,a_k\in G^\circ$ and $b_1,\ldots,b_l\in G$ have the property that any two elements in $\{a_1,\ldots,a_k,b_1,\ldots,b_l\}$ commute.
Let $W\in\Hardy$ satisfy $1\prec W(t)\ll t$, and suppose $g_1,\ldots,g_m\in \Hardy$ have the following properties:
\begin{enumerate}
[label=(\Alph{enumi}),ref=(\Alph{enumi}),leftmargin=*]
\item
\label{itm_G0}
$\{g_1,\ldots,g_m\}$ satisfies Property \ref{property_P_W};
\item
\label{itm_G2}
$|g(t)-p(t)|\to\infty$ for any $g\in\sp(g_1,\ldots,g_m)$ and $p\in\R[t]$.
\end{enumerate}
Let $\lambda_{1,1},\ldots,\lambda_{k,m}\in\R$, let $p_1,\ldots,p_l\in\R[t]$, define $\phi_{i}(t_1,\ldots,t_m)\coloneqq \sum_{j=1}^m \lambda_{i,j} t_j$, and assume that $p_i(\N)\subset\dom(b_i)$ for all $i\leq l$ and that the set
$$
\Big\{a_1^{\phi_1(t_1,\ldots,t_m)}\cdot \ldots\cdot a_k^{\phi_k(t_1,\ldots,t_m)}\cdot b_1^{p_1(n)}\cdot\ldots\cdot b_l^{p_l(n)}\Gamma: t_1,\ldots, t_m\in\R,~n\in\Z\Big\}
$$
is dense in $X$.

Then the sequence
$$
\big(a_1^{\phi_1(g_1(n),\ldots,g_m(n))}\cdot \ldots\cdot a_k^{\phi_k(g_1(n),\ldots,g_m(n))}\cdot b_1^{p_1(n)}\cdot \ldots\cdot b_l^{p_l(n)}\Gamma\big)_{n\in\N}
$$
is uniformly distributed in $X$ with respect to $W$-averages.
\end{Corollary}

\begin{proof}
According to \cref{thm_D} it suffices to show that the projection of the sequence
$$
\big(a_1^{\phi_1(g_1(n),\ldots,g_m(n))}\cdot \ldots\cdot a_k^{\phi_k(g_1(n),\ldots,g_m(n))}\cdot b_1^{p_1(n)}\cdot \ldots\cdot b_l^{p_l(n)}\Gamma\big)_{n\in\N}
$$
onto the maximal factor torus $[G^\circ,G^\circ]\backslash X$ is uniformly distributed with respect to $W$-averages there.
Since $X=G/\Gamma$ is connected, we have that $G^\circ \Gamma=\Gamma$.
This implies (see \cite[Subsection 2.6]{Leibman05a}) that there exist $e_1,\ldots,e_s\in G^\circ$ and $q_1,\ldots,q_s\in\Z[t]$ such that
$$
b_1^{p_1(n)}\cdot \ldots\cdot b_l^{p_l(n)}\Gamma~=~e_1^{q_1(n)}\cdot \ldots\cdot e_s^{q_s(n)}\Gamma.
$$
It thus suffices to show that the projection of the sequence
\begin{equation}
\label{eqn_new_seq_bp_1}
\big(a_1^{\phi_1(g_1(n),\ldots,g_m(n))}\cdot \ldots\cdot a_k^{\phi_k(g_1(n),\ldots,g_m(n))}\cdot e_1^{q_1(n)}\cdot \ldots\cdot e_s^{q_s(n)}\Gamma\big)_{n\in\N}
\end{equation}
is uniformly distributed with respect to $W$-averages in the maximal factor torus $[G^\circ,G^\circ]\backslash X$.
Note that since $X$ is connected, the embedding $G^\circ \hookrightarrow G$ descents to an isomorphism between the finite dimensional torus $G^\circ/([G^\circ,G^\circ]\Gamma)$ and the quotient $[G^\circ,G^\circ]\backslash X$.
The down-side of replacing $b_1^{p_1(n)}\cdot \ldots\cdot b_l^{p_l(n)}\Gamma$ with $e_1^{q_1(n)}\cdot \ldots\cdot e_s^{q_s(n)}\Gamma$ is that the $e_i$s and $a_j$s are not necessarily commuting. However, after projecting onto $[G^\circ,G^\circ]\backslash X$, this doesn't matter. On the flip side, the advantage of replacing $b_1^{p_1(n)}\cdot \ldots\cdot b_l^{p_l(n)}\Gamma$ with $e_1^{q_1(n)}\cdot \ldots\cdot e_s^{q_s(n)}\Gamma$ is that the action of the sequence \eqref{eqn_new_seq_bp_1} remains an action by translation when passed through the isomorphism that identifies $[G^\circ,G^\circ]\backslash X$ with $G^\circ/([G^\circ,G^\circ]\Gamma)$ because its generators $a_1,\ldots,a_k,e_1,\ldots, e_s$ belong to $G^\circ$. Working within translations on $G^\circ/([G^\circ,G^\circ]\Gamma)$ is convenient since it enables us to use uniform distribution results for finite dimensional tori from \cref{sec_prelims}.

Finally, the fact that the projection of
$$
\big(a_1^{\phi_1(g_1(n),\ldots,g_m(n))}\cdot \ldots\cdot a_k^{\phi_k(g_1(n),\ldots,g_m(n))}\cdot e_1^{q_1(n)}\cdot \ldots\cdot e_s^{q_s(n)}\Gamma\big)_{n\in\N}
$$
onto $G^\circ/([G^\circ,G^\circ]\Gamma)$ is uniformly distributed follows from the fact that the projection of the set
$$
\overline{\Big\{a_1^{\phi_1(t_1,\ldots,t_m)}\cdot \ldots\cdot a_k^{\phi_k(t_1,\ldots,t_m)}\cdot e_1^{q_1(n)},\cdot \ldots\cdot e_s^{q_s(n)}\Gamma: t_1,\ldots, t_m\in\R,~n\in\Z\Big\}}
$$
is dense, together with properties \ref{itm_G0} and \ref{itm_G2}, the Weyl criterion and \cref{prop_Bosh_W-averages}.
\end{proof}

\begin{proof}[Proof of \cref{thm_nilequidistribution_2_1_new}]
Replacing $X$ with the subnilmanifold $\overline{\{b^t\Gamma:t\in\R\}}$ we may assume that the set $\{b^t\Gamma:t\in\R\}$ is dense in $X$.

It follows directly from \cref{thm_E} that there exist $q\in\N$ and connected sub-nilmanifolds $Y_0 , Y_1, \ldots, Y_{q-1}$ of $X^k$ such that for every $r\in\{0,1,\ldots,q-1\}$ the sequence
$$
\big(\big(b^{f_1(qn+r)},b^{f_2(qn+r)},\dots,b^{f_k(qn+r)}\big)\Gamma^k\big)_{n\in\N}
$$
is uniformly distributed with respect to the Haar measure $\mu_{Y_r}$ of $Y_r$ and $W$-averages.
Thus, if we set $\nu\coloneqq (\mu_{Y_0}+\ldots+\mu_{Y_{q-1}})/q$ then it follows from \cref{lemma_WaveragesoverAPs} that the sequence $\big(\big(b^{f_1(n)},\dots,b^{f_k(n)}\big)\Gamma^k\big)_{n\in\N}$ is uniformly distributed with respect to $\nu$ and $W$-averages.
This proves the first part of \cref{thm_nilequidistribution_2_1_new}.

Next we prove that if condition \ref{itm_cond_1} of Theorem~\ref{thm_main_combinatorial} holds, then $1_X^k\in\supp\nu$.
Let $q_1,\dots,q_\ell\in\Z[x]$ be jointly intersective polynomials such that $
\poly(f_1,\dots,f_k)\subset \sp(q_1,\dots,q_\ell)
$.
We claim that there exists $r\in\{0,\dots,q-1\}$ such that the polynomials $\tilde q_i:n\mapsto q_i(qn+r)$, $i\in\{1,\dots,\ell\}$, are jointly intersective. Indeed, since $q_1,\dots,q_\ell$ are jointly intersective, for any $w\in\N$ there exists $n_w\in\N$ such that $q_i(n_w)\equiv 0\bmod w!$ for all $i\in\{1,\dots,\ell\}$. By the pigeonhole principle, there exists $r\in \{0,\dots,q-1\}$ such that $n_w\equiv r\bmod q$ for infinitely many $w$. With this choice of $r$, it follows that the polynomials $\tilde q_1,\ldots, \tilde q_\ell$ have a common zero modulo $w!$ for infinitely many $w\in\N$. But this implies that $\tilde q_1,\ldots, \tilde q_\ell$ have a common zero modulo any number, proving that they are jointly intersective.
We shall show that $1_{X^k}\in Y_r$.

Define, for all $i\in\{1,\ldots,k\}$, the function $\tilde f_i(t)\coloneqq f_i(qt+r)$, and let
$\mathcal{P}\,\coloneqq \poly(\tilde f_1,\dots,\tilde f_k)$.
Observe that $\mathcal{P}\subset\sp(\tilde q_1,\dots,\tilde q_\ell)$.
Applying \cref{lem_simple_normal_form_new_new} to $\tilde f_1,\ldots,\tilde f_k$ we find disjoint $\mathcal{I},\mathcal{J}\subset\{1,\ldots,k\}$ with $\mathcal{I}\cup\mathcal{J}=\{1,\ldots,k\}$, $\{\lambda_{i,j}: i\in\mathcal{I}, j\in\mathcal{J}\}\subset\R$, and $\{p_i: i\in\mathcal{I}\}\subset\mathcal{P}$ such that conditions \ref{itm_Z1} and \ref{itm_Z2} of \cref{lem_simple_normal_form_new_new} are satisfied. After reordering $\tilde f_1,\ldots,\tilde f_k$ if necessary, we can assume that $\mathcal{J}=\{1,\ldots,l\}$ and $\mathcal{I}=\{l+1,\ldots,k\}$ for some $l\in\{0,1,\ldots,k\}$ (where $l=0$ corresponds to the case $\mathcal{J}=\emptyset$ and $l=k$ to the case $\mathcal{I}=\emptyset$). For $i\in\{l+1,\ldots,k\}$ define
$$
\phi_i(t_1,\ldots,t_l)\, =\, \sum_{j=1}^l \lambda_{i,j} t_j\quad\text{and}\quad f_i^*(t)=\phi_i\big(\tilde f_1(t),\dots,\tilde f_l(t)\big)+p_i(t)
$$
and observe that $|f_{i}(t)-f_i^*(t)|\to 0$ as $t\to\infty$.
Therefore
\begin{equation}
\label{eqn_distance_goes_to_0}
d\Big(\big(b^{\tilde f_1(n)},\dots,b^{\tilde f_k(n)}\big)\Gamma^k,\big(b^{\tilde f_1(n)},\ldots, b^{\tilde f_l(n)}, b^{f_{l+1}^*(n)},\dots,b^{f_k^*(n)}\big)\Gamma^k\Big)\to0
\end{equation}
as $n\to\infty$.
Since the sequence $((b^{\tilde f_1(n)},\dots,b^{\tilde f_k(n)})\Gamma^k)_{n\in\N}$ is uniformly distributed with respect to $\mu_{Y_r}$ and $W$-averages, it follows from \eqref{eqn_distance_goes_to_0} that the same is true for the sequence $\Big(\big(b^{\tilde f_1(n)},\ldots, b^{\tilde f_l(n)}, b^{f_{l+1}^*(n)},\dots,b^{f_k^*(n)}\big)\Gamma^k\Big)_{n\in\N}$.
In view of \cref{thm_nilequidistribution_2_2}, we therefore must have that $Y_r=\overline{\{y_{t,n}:t\in\R^l,n\in\N\}}$, where for $t=(t_1,\dots,t_l)\in\R^l$ and $n\in\N$ we define
\begin{equation}\label{eq_yetanothereq}
y_{t,n}=\big(b^{t_1},\ldots, b^{t_l}, b^{\phi_{l+1}(t_1,\ldots,t_l)+p_{l+1}(n)},\dots,b^{\phi_{k}(t_1,\ldots,t_l)+p_{k}(n)}\big)\Gamma^k.
\end{equation}
It now follows from \cite[Proposition 2.3]{BLL08} that the sequence $(y_{0,n})_{n\in\N}$ has $1_{X^k}$ as an accumulation point.

Finally, we assume that condition \ref{itm_cond_2} of Theorem~\ref{thm_main_combinatorial} holds.
We will show that in this case $Y_0=X^k$ (and in particular $q=1$).
Following the same procedure as above but with $r=0$, we conclude that $Y_0=\overline{\{y_{t,n}:t\in\R^l,n\in\N\}}$ where $y_{t,n}$ is given by \eqref{eq_yetanothereq}.

We claim that for every fixed $t\in\R^l$, the sequence $(y_{t,n})_{n\in\N}$ is dense in the set $\{b^{t_1}\Gamma\}\times\cdots\times\{b^{t_l}\Gamma\}\times X^{k-l}$.
Indeed, this follows from \cite[Corollary 1.7]{Richter20arXiv} together with condition \ref{itm_cond_2} of Theorem~\ref{thm_main_combinatorial} (which implies that any integer linear combination $p$ of the polynomials $p_{l+1},\dots,p_k$ must satisfy $p-p(0)\notin\Q[x]$) and our assumption that $\{b^t\Gamma:t\in\R\}$ is dense in $X$.
Therefore it follows that the set $\{y_{t,n}:t\in\R^l,n\in\N\}$ is dense in $X^k$ and hence that $Y_0=X^k$ as desired.
\end{proof}

\section{Open questions}
\label{sec_questions}
In this section we collect some pertinent open questions and conjectures.

Let $f_1,\dots,f_k$ be linearly independent functions of the form $f_i(t)=a_1t^{c_1}+\cdots+a_dt^{c_d}$ where $a_i,c_i\in\R$, $c_i>0$.
If all the $f_i$ are integer polynomials, then we know from \cite{FK06} that for any totally ergodic system $(X,\B,\mu,T)$ and any $h_1,\dots,h_k\in L^\infty(X)$,
\begin{equation}\label{eq_conjecture1}
    \lim_{N\to\infty}\frac1N\sum_{n=1}^NT^{f_1(n)}h_1\cdot \ldots\cdot T^{f_k(n)}h_k=\int_Xh_1\d\mu\cdot\ldots\cdot\int_Xh_k\d\mu.
\end{equation}
On the other hand, if all the $c_i$ are non-integers, then it follows from \cref{example_rightlimit} that for any ergodic system $(X,\B,\mu,T)$ and any $h_1,\dots,h_k\in L^\infty(X)$,
\begin{equation}\label{eq_conjecture7}
    \lim_{N\to\infty}\frac1N\sum_{n=1}^NT^{[f_1(n)]}h_1\cdot \ldots\cdot T^{[f_k(n)]}h_k=\int_Xh_1\d\mu\cdot\ldots\cdot\int_Xh_k\d\mu.
\end{equation}

The following conjecture expands on the above observations and is supported by multiple results and conjectures involving the notion of \emph{joint ergodicity}, including \cite[Theorem 2.3]{Frantzikinakis15b}, \cite[Theorem 1.7]{Frantzikinakis21arXiv}, \cite[Conjecture 1.5]{DKW21}, \cite{BB84}, \cite{BLS16}.

\begin{Conjecture}
\label{conj_6-1-new}
Let $f_1,\dots,f_k$ be functions of the form $f_i(t)=a_1t^{c_1}+\cdots+a_dt^{c_d}$ where $a_i,c_i\in\R$, $c_i>0$. Let $(X,\B,\mu,T_1,\ldots,T_k)$ be a measure preserving system with $k$ commuting (and invertible) measure-preserving transformations.
Suppose
\begin{itemize}
\item
for all $h\in L^\infty(X)$ and $i\neq j$ one has
\[
\frac1N\sum_{n=1}^N T_i^{[f_i(n)]}T_j^{-[f_j(n)]} h \xrightarrow[N\to\infty]{}\int_Xh\d\mu,
\]
where the convergence takes place in $L^2(X)$.
\item
for all $H\in L^\infty(X^k)$ one has
\[
\lim_{N\to\infty}\frac1N\sum_{n=1}^N \big(T_1^{[f_1(n)]}\times\ldots\times T_k^{[f_k(n)]}\big) H=\int_{X^k} H\d\mu^k
\]
where the convergence takes place in $L^2(X^k)$.
\end{itemize}
Then for all $h_1,\dots,h_k\in L^\infty(X)$ we have
\[
\lim_{N\to\infty}\frac1N\sum_{n=1}^N T_1^{[f_1(n)]}h_1\cdot \ldots\cdot T_k^{[f_k(n)]}h_k=\int_Xh_1\d\mu\cdot\ldots\cdot\int_Xh_k\d\mu
\]
where the convergence takes place in $L^2(X)$.
\end{Conjecture}
\begin{Remark}
It may well be the case that \cref{conj_6-1-new} is true for \emph{any} functions $f_1,\ldots,f_k$ in a Hardy field $\Hardy$.
\end{Remark}

\begin{Conjecture}
  Let $f_1,\ldots,f_k$ be functions of polynomial growth from a Hardy field $\Hardy$, let $W$ be a compatible weight, let $(X,\B,\mu,T)$ be an invertible measure preserving system, let $h_0,\ldots,h_k\in L^\infty(X)$ and consider the multicorrelation sequence
\begin{equation*}
\alpha(n):=\int_Xh_0\cdot T^{[f_1(n)]}h_1\cdots T^{[f_k(n)]}h_k\d\mu.
\end{equation*}
Then there exists a nilmanifold $Y=G/\Gamma$, a continuous function $F\in C(Y)$, a point $y\in Y$ and $a_1,\dots,a_k\in G$ such that
$$\alpha(n)=F\left(a_1^{[f_1(n)]}\cdots a_k^{[f_k(n)]}y\right)+\nu(n),\ \ n\in\N$$
where $\nu$ satisfies
$$\lim_{N\to\infty}\frac1{W(N)}\sum_{n=1}^Nw(n)\big|\nu(n)\big|=0.$$
\end{Conjecture}

Our notion of compatibility between a tuple of functions $f_1,\dots,f_k\in\Hardy$ and a weight function $W$ hinges on the Property \ref{property_P_W}.
This relationship is necessary to prove Theorem~\ref{thm_main} because of the reliance on the results from \cite{Richter20arXiv}. \
However, it is possible that a weaker notion of compatibility is sufficient.
Given a Hardy field $\Hardy$ and $f_1,\dots,f_k,W\in\Hardy$ we define the property
\begin{enumerate}
[label=\textbf{Property (WP):},ref=(WP),leftmargin=*]
\item\label{property_weakP_W}

For all $f\in\sp(f_1,\ldots,f_k)$ and $p\in\R[t]$ either $|f(t)-p(t)|\ll 1$ or $|f(t)-p(t)| \succ \log(W(t))$.
\end{enumerate}

\begin{Conjecture}\label{conjecture_weakP}
 Theorem~\ref{thm_main} holds if Property \ref{property_P_W} gets replaced with Property \ref{property_weakP_W}.
\end{Conjecture}
We remark that \cref{conjecture_weakP} implies \cref{conj_frantzi_2}.

Condition \ref{itm_cond_1} in Theorem~\ref{thm_main_combinatorial} is somewhat complicated.
In a way this is inevitable if we want to allow families of jointly intersective polynomials.
However, there is a simpler, more natural, and slightly weaker condition which might be sufficient to imply the conclusion.
\begin{Question}
  Can one replace condition \ref{itm_cond_1} in Theorem~\ref{thm_main_combinatorial} with the weaker assumption that the collection of polynomials
$$
\poly(f_1,\ldots,f_k)\cap \Z[t]\,=\,\Big\{q\in\Z[t]: \exists f\in\sp(f_1,\dots,f_k)\, \text{with}\, \lim_{t\to\infty}|f(t)-q(t)|=0\Big\}
$$
is jointly intersective?
\end{Question}

One of the motivations for this paper was to expand on our previous work from \cite{BMR20} which revealed a new phenomenon pertaining to non-polynomial functions, say, from a Hardy field  (see \cref{thm_thick_szemeredi} above).
As a corollary of \cref{thm_thick_szemeredi}, if $f\in\Hardy$ and $t^{\ell-1}\prec f(t)\prec t^\ell$ for some $\ell\in\N$, then for every $E\subset\N$ with $\bar d(E)>0$ the set
\begin{equation}\label{eq_returntimes}
  R=R_f:=\Big\{n\in\N:E\cap\big(E-\big\lfloor f(n)\big\rfloor\big)\neq\emptyset\Big\}
\end{equation}
is \emph{thick}, i.e. contains arbitrarily long intervals.
This stands in contrast with the case when $f$ is a polynomial, in which case $R(A,f)$ is in general not thick, but is always \emph{syndetic}, i.e. it has bounded gaps.
This difference is all the more striking since syndeticity and thickness are complementary notions (a set is syndetic if and only if its complement is not thick).
Regarding \eqref{eq_returntimes} we have a dichotomy:

$$R_f\text{ is }\begin{cases}
  \text{ thick }&\text{ if }t^{\ell-1}\prec f(t)\prec t^\ell\text{ for some }\ell\in\N\\
  \text{ syndetic }&\text{ otherwise}.
\end{cases}$$

Corollary A4 in the Introduction implies that if $f_1,\dots,f_k$ belong to a Hardy field and satisfy a ``non-polynomiality condition'' then the intersection $R_{f_1}\cap \cdots\cap R_{f_k}$ is thick.
While we can not replace the condition with the more natural $\poly(f_1,\dots,f_k)=\emptyset$, as shown in \cref{example_4}, there is an intermediate condition which might be sufficient.
\begin{Question}
  Does Corollary A4 still hold if $\Hsp(f_1,\dots,f_k)$ is replaced with the set
  $$\big\{c_1 f_1^{(m_1)}(t)+\ldots+c_k f_k^{(m_k)}(t): c_1,\ldots,c_k\in\Z,\, m_1,\ldots,m_k\in\N\cup\{0\}\big\}?$$
\end{Question}
In particular, we don't know if letting $f_1(t)=t^{3/2}$ and $f_2(t)=\alpha t^{3/2}+t$, where $\alpha\in\R\setminus\Q$, the intersection $R_{f_1}\cap R_{f_2}$ is thick.


\bibliographystyle{aomalphanomr}
\bibliography{mynewlibrary.bib}



\bigskip
\footnotesize
\noindent
Vitaly Bergelson\\
\textsc{The Ohio State University}\par\nopagebreak
\noindent
\href{mailto:bergelson.1@osu.edu}
{\texttt{bergelson.1@osu.edu}}

\bigskip
\footnotesize
\noindent
Joel Moreira\\
\textsc{University of Warwick}\par\nopagebreak
\noindent
\href{mailto:joel.moreira@warwick.ac.uk}
{\texttt{joel.moreira@warwick.ac.uk}}

\bigskip
\footnotesize
\noindent
Florian K.\ Richter\\
\textsc{École Polytechnique Fédérale de Lausanne (EPFL)}\par\nopagebreak
\noindent
\href{mailto:f.richter@epfl.ch}
{\texttt{f.richter@epfl.ch}}

\end{document}